\documentclass[12pt]{article}
\usepackage{amsmath,amssymb,amsthm,mathtools,enumitem}
\usepackage{color}
\usepackage{geometry}
\newtheorem{Theorem}{Theorem}[section]

\newtheorem{Lemma}{Lemma}[section]

\begin{document}

\def\eR{\mathbf{R}}
\def\Rd{{\eR}^d}
\def\Rdd{{\eR}^{d\times d}}
\def\Rdsym{{\eR}^{d\times d}_{sym}}
\def\eN{\mathbf{N}}
\def\eZ{\mathbf{Z}}

\def\dd{\mbox{d}}
\newcommand{\essinf}{\operatorname{ess\,inf}}
\newcommand{\esssup}{\operatorname{ess\,sup}}
\newcommand{\supp}{\operatorname{supp}}
\newcommand{\dist}{\operatorname{dist}}
\newcommand{\intr}{\operatorname{int}}
\def\d{\; \mathrm{d}}
\def\dx{\; \mathrm{d}x}
\def\dy{\; \mathrm{d}y}
\def\dz{\; \mathrm{d}z}
\def\diff{\mathsf{d}}
\def\div{\mathop{\mathrm{div}}\nolimits}
\def\diam{\mathrm{diam}}

\def\Aeps{\mathbf{A}^\varepsilon}
\def\ueps{\mathbf{u}^\varepsilon}
\def\boldp{\mathbf{p}}
\def\bolds{\mathbf{s}}
\def\boldu{\mathbf{u}}
\def\boldv{\mathbf{v}}
\def\boldw{\mathbf{w}}
\def\boldz{\mathbf{z}}
\def\boldA{\mathbf{A}}
\def\boldB{\mathbf{B}}
\def\boldD{\mathbf{D}}
\def\boldF{\mathbf{F}}
\def\boldI{\mathbf{I}}
\def\boldO{\mathbf{O}}
\def\boldP{\mathbf{P}}
\def\boldT{\mathbf{T}}
\def\boldU{\mathbf{U}}
\def\boldUeps{\mathbf{U}^\varepsilon}
\def\boldV{\mathbf{V}}
\def\boldW{\mathbf{W}}
\def\boldZ{\mathbf{Z}}
\def\bzero{\mathbf{0}}
\def\uk{\boldu^{k}}
\def\boldAk{\boldA^k}
\def\balpha{\boldsymbol{\alpha}}
\def\bbeta{\boldsymbol{\beta}}
\def\boldzeta{\boldsymbol{\zeta}}
\def\boldeta{\boldsymbol{\eta}}
\def\boldchi{\boldsymbol{\chi}}
\def\boldxi{\boldsymbol{\xi}}
\def\bphi{\boldsymbol{\varphi}}
\def\bpsi{\boldsymbol{\psi}}
\def\tbpsi{\tilde{\bpsi}}
\def\WCon{\xrightharpoonup{\hphantom{2-s}}}
\def\WSCon{\xrightharpoonup{\raisebox{0pt}[0pt][0pt]{\hphantom{\(\scriptstyle{2-s}\)}}}^*}
\def\WTSSCon{\xrightharpoonup{\raisebox{0pt}[0pt][0pt]{\(\scriptstyle{2-s}\)}}^*}
\def\STSCon{\xrightarrow{2-s}}
\def\SCon{\xrightarrow{\raisebox{0pt}[0pt][0pt]{\hphantom{\(\scriptstyle{2-s}^*\)}}}}
\def\ModConv#1{\xrightarrow{\mathmakebox[1.5em]{#1}}}
\def\ModConvM{\ModConv{M}}
\def\NabOv{\overline\nabla}

\title{Homogenization of nonlinear elliptic systems in nonreflexive Musielak-Orlicz spaces}
\author{Miroslav Bul\'\i\v cek$^1$, Piotr Gwiazda$^{2,3}$, Martin Kalousek$^2$, Agnieszka \'Swierczewska-Gwiazda$^3$}
\date{}
\maketitle
\abstract{We study the homogenization process for families of strongly nonlinear elliptic systems with the homogeneous Dirichlet boundary conditions. The growth and the coercivity of the elliptic operator is assumed to be indicated by a general  inhomogeneous anisotropic $\mathcal{N}$--function $M$, which may also depend on the spatial variable, i.e., the homogenization process will change the underlying function spaces and the nonlinear elliptic operator at each step. The problem of homogenization of nonlinear elliptic systems has been solved  for the $L^p-$setting with restrictions either on constant exponent or variable exponent that is assumed to be additionally log-H\"older continuous. These results correspond to a very particular case of $\mathcal{N}$--functions satisfying both $\Delta_2$ and $\nabla_2$--conditions. We show that for general $M$ satisfying a condition of log-H\"older type continuity,  one can provide a rather general theory without any assumption on the validity of neither $\Delta_2$ nor $\nabla_2$--conditions.}

\footnotetext[1]{Charles University, Faculty of Mathematics and
Physics, Mathematical Institute, Sokolovsk\'{a}~83,
186~75~Prague~8, Czech~Republic}
\footnotetext[2]{Institute of Mathematics, Polish Academy of Sciences}
\footnotetext[3]{Faculty of Mathematics, Informatics and Mechanics
University of Warsaw, Banacha 2, 02-097 Warszawa, Poland}

\section{Introduction}

Our primary interest is to study the behaviour of the following system as $\varepsilon \to 0_+$:
\begin{equation}\label{StudPr}
\begin{aligned}
\div \boldA\left(\frac{x}{\varepsilon},\nabla\ueps\right)& =\div\boldF &&\text{ in }\Omega,\\
\ueps&=0 &&\text{ on }\partial\Omega,
\end{aligned}
\end{equation}
where $\Omega\subset\Rd,d\geq 2$ is a bounded domain and $\ueps\!:\Omega \to \eR^N$ with $N\in \mathbf{N}$ is an unknown and  $\boldF\!:\Omega\rightarrow\eR^{d\times N}$ and $\boldA\!: \Rd\times \eR^{d\times N}\rightarrow \eR^{d\times N}$ are given. The operator $\boldA$ is periodic with respect to the first variable and  strongly nonlinear with respect to the second variable  with the growth prescribed by a spatially inhomogeneous and in general anisotropic $\mathcal{N}$--function. We aim to study the most general class of operators, $\boldA$, which will however lead to the problems with the so-called nonstandard growth conditions, i.e., conditions given via general Musielak--Orlicz spaces that may vary with changing parameter~$\varepsilon$ and may heavily depend on the spatial variable.

The studies on homogenization of elliptic equations go back to the fundamental lecture of Tartar~\cite{Tartar2} and also consequent works \cite{Tartar,SP,OZ82}  and are of the highest interest among  the properties of elliptic systems with  periodic structure. The homogenization process was also the starting point for developing the two-scale convergence technique, which was introduced by  Allaire \cite{Al92} and later generalized to the framework of more general operators in \cite{ZP11}. Following ideas presented in  \cite{ZP11}, the suggested homogenization process, i.e., letting $\varepsilon \to 0$ in \eqref{StudPr}, one expects that $\boldu^{\varepsilon} \to \boldu$, where $\boldu$ is a solution to the following nonlinear elliptic problem with the nonlinear operator independent of a spatial variable, i.e.,
\begin{equation}\label{HPr}
	\begin{alignedat}{2}	
			\div\hat\boldA(\nabla \boldu)&=\div\boldF &&\text{ in }\Omega,\\
			\boldu&=0 &&\text{ on }\partial\Omega.
	\end{alignedat}
\end{equation}
Here, we  denoted $Y:=(0,1)^d$ and defined the operator $\hat\boldA$  as
\begin{equation*}
	\hat\boldA(\boldxi):=\int_Y\boldA(y,\boldxi+\nabla\boldw_{\boldxi}(y))\dy,
\end{equation*}
and for any $\boldxi\in \eR^{d \times N}$, the function $\boldw_{\boldxi}:\eR^d \to \eR^N$ is the solution of the cell problem, i.e., $\boldw_{\boldxi}$ is  $Y$-periodic and solves in the sense of distributions
\begin{equation*}
	\div\boldA(y,\boldxi+\nabla \boldw_{\boldxi}(y))=0\text{ in }Y.
\end{equation*}

Our  main goal  is to  justify rigorously the above mentioned heuristic procedure based on~\cite{ZP11}, where the setting of non--standard growth conditions  of the operator $\boldA$ was  considered.  The authors studied the case  of variable exponent $p(x/\varepsilon)$. To identify an elliptic operator in the homogenized problem, i.e., the operator $\hat\boldA$ appearing in \eqref{HPr}, they  applied a variant of the compensated compactness argument. For such an approach, one however requires that the Helmholtz-like decomposition holds for functions belonging to the involved function spaces, which in this case were the variable  exponent Lebesgue spaces with log-H\"{o}lder continuous exponent.  It has to be pointed out at the very beginning that a decomposition of a similar type does not hold  for problems with growth conditions considered here. The main novelty of the paper is performing  the limit procedure $\varepsilon \to 0$ in \eqref{StudPr} even though  the underlying function spaces do not allow for using the methods based on the Helmoltz decomposition. The results, for various cases,  are summarized in Theorem~\ref{Thm:Main}. It is worth noticing that we do not require that the $\mathcal{N}$--function $M$, corresponding to the operator $\boldA$,  satisfies $\Delta_2$ or  $\nabla_2$--condition, and the only  assumption  is a certain form of  log-H\"{o}lder continuity with respect to the spatial variable of the $\mathcal{N}$--function $M$. In particular, the key result of the paper is an introduction of a completely new technique, which is not based on the Hemoltz-like decomposition, but rather deals with the combination of two--scale limit and the so--called modular convergence.

We first formulate certain minimal assumptions on the operator $\boldA$, that will be used in what follows:
\begin{enumerate}[label=(A\arabic*)]
	\item\label{AO} $\boldA$ is a Carath\'eodory mapping, i.e., $\boldA(\cdot,\boldxi)$ is measurable for any $\boldxi\in \eR^{d\times N}$ and $\boldA(y,\cdot)$ is continuous for a.a. $y\in\Rd$,
	\item $\boldA$ is $Y-$periodic, i.e., periodic in each argument $y_i,i=1,\ldots,d$ with the period $1$,
	\item\label{ATh} There exists an ${\mathcal N}$--function $M\!:\Rd\times \eR^{d\times N}\rightarrow[0,\infty)$ and a constant $c>0$ such that for a.a. $y\in Y$ and all $\boldxi\in \eR^{d\times N}$ there holds\footnote{  Note that the condition could be formulated more generally, i.e., $\boldA(y,\boldxi)\cdot\boldxi\geq c(M(y,\boldxi)+M^*(y,\boldA(y,\boldxi)))-k(y)$ for some integrable function $k$.  For readability we omit this generality here setting $k\equiv 0$, however such case could easily be treated, see e.g.~\cite{GSG08}. }
	\begin{equation*}
		\boldA(y,\boldxi)\cdot\boldxi\geq c(M(y,\boldxi)+M^*(y,\boldA(y,\boldxi))),
	\end{equation*}
	\item\label{AF} For all $\boldxi,\boldeta\in\eR^{d\times N}$ such that $\boldxi\neq\boldeta$ and a.a. $y\in Y$, we have
	\begin{equation*}
		(\boldA(y,\boldxi)-\boldA(y,\boldeta))\cdot(\boldxi-\boldeta)> 0.
	\end{equation*}
\end{enumerate}

Before we introduce the assumption on the function $M$, we shall denote a particular  covering of $Y$ by $d$--dimensional cubes  $Q_j^\delta$. More precisely,  a family $\{Q_j^\delta\}_{j=1}^{N^\delta}$ consists of closed cubes of edge $2\delta$ such that $\intr Q_j^\delta\cap\intr Q_i^\delta=\emptyset$ for $i\neq j$ and $Y\subset \bigcup_{j=1}^{N^\delta} Q_j^\delta$. Moreover, for each cube $Q_j^\delta$ we define the cube $\tilde{Q}_j^\delta$ centered at the same point and with parallel corresponding edges of length $4\delta$.
Finally, we impose the following conditions on $M$:
\begin{enumerate}[label=(M\arabic*)]
	\item\label{MO} $M$ is and $\mathcal{N}$--function that is $Y-$periodic with respect to the first  variable,
		\item\label{MTwo} there exist ${\mathcal N}$--functions $m_1,m_2:[0,\infty)\rightarrow[0,\infty)$ such that for all $\boldxi\in \eR^{d\times N}$ and all $y\in Y$
	\begin{equation*}
			m_1(|\boldxi|)\leq M(y,\boldxi)\leq m_2(|\boldxi|),
	\end{equation*}
	\item \label{M4}
	 there exist constants $C, D, E>0$ and $G\geq 1$ such that for all $y\in Q_j^\delta$ and all $\boldxi\in\eR^{d\times N}$ we have
\begin{equation*}
	\frac{M(y,\boldxi)}{(M^\delta_j)^{**}(\boldxi)}\leq C\max\{|\boldxi|^{-\frac{D}{\log(E\delta)}},G^{-\frac{D}{\log(E\delta)}}\}
\end{equation*}
where $\delta<\delta_0$ ($\delta_0$ is chosen in such a way that $E\delta_0\leq\frac{1}{2}$),
\begin{equation}\label{MDeltajDef}
	M_j^\delta(\boldxi):=\inf_{y\in \tilde{Q}_j^\delta}M(y,\boldxi).
\end{equation}
and $(M^\delta_j)^{**}$ is the biconjugate of $M^\delta_j$.
\end{enumerate}

Having stated the assumptions, we introduce the main result of the paper.
\begin{Theorem}\label{Thm:Main}
Let $\boldA$ satisfy \ref{AO}--\ref{AF}, the ${\mathcal N}$--function $M$ satisfy \ref{MO}--\ref{M4},
\begin{equation}\label{Assumption:F}
\boldF\in L^{\infty}(\Omega; \eR^{d\times N})
\end{equation}
and for any $\varepsilon>0$ let $\ueps$  be a unique solution of the problem~\eqref{StudPr}. Then for an arbitrary sequence $\{\varepsilon_j\}_{j=1}^\infty$ such that $\varepsilon_j\rightarrow 0$ as $j\rightarrow\infty$, we have the following convergence result
\begin{equation*}
	\boldu^{\varepsilon_j}\rightharpoonup \boldu\text{ in }W^{1,1}_0(\Omega;\eR^N),
\end{equation*}
where $\boldu^{\varepsilon_j}$ is the sequence of solutions solving~\eqref{StudPr} with $\varepsilon=\varepsilon_j$ and  $\boldu$ is a unique solution to~\eqref{HPr}, provided that one of the following conditions holds:
\begin{enumerate}[label=(C\arabic*)]
\item\label{MTw} The set $\Omega$ is star--shaped.
\item\label{MTw2} The set $\Omega$ is Lipschitz and we have the single equation, i.e., $N=1$.
\item\label{EMD} The embedding $W^{1,m_1}(\Omega)\hookrightarrow L^{m_2}(\Omega)$ holds.
\end{enumerate}
\end{Theorem}
We would like to emphasize here, that this is the first result that does not require validity of neither $\Delta_2$ nor $\nabla_2$--condition and relies only on the assumption of log-H\"{o}lder continuity type.

It is also remarkable here, that we require a kind of implicit assumption \ref{M4}, which maybe very hard to check for functions $M$ with complicated structure. Moreover, following the log-H\"{o}lder continuity assumption in the variable exponent Lebesgue spaces, one would expect that the following condition could be sufficient:
\begin{enumerate}[label=(M\arabic*)] \setcounter{enumi}{3}
	\item\label{MTh}
	there exist constants $A>0$ and $B\geq 1$ such that for all $y_1,y_2\in Y$ with $|y_1-y_2|\leq \frac{1}{2}$ and all $\boldxi\in \eR^{d\times N}$ we have
	\begin{equation*}
		\frac{M(y_1,\boldxi)}{M(y_2,\boldxi)}\leq \max\{|\boldxi|^{-\frac{A}{\log|y_1-y_2|}},B^{-\frac{A}{\log|y_1-y_2|}}\}.
	\end{equation*}
\end{enumerate}
Clearly, \ref{M4} directly implies \ref{MTh}. However, it is not known  whether also the opposite implication holds true for general functions $M$. Nevertheless, for examples we have in mind, the assumptions \ref{M4} and \ref{MTh} are in fact equivalent and the results of Theorem~\ref{Thm:Main} are valid.

The first example is the radially symmetric function $M$, i.e.,
$$
 M(y,\boldxi)=\tilde M(y,|\boldxi|),
$$
where $\tilde M$ is an $\mathcal{N}$-function satisfying \ref{MTh}. Then, one can show that $M$ automatically satisfies also \ref{M4}. The detailed proof of this observation is provided in the Appendix~\ref{Ape1}, see Lemma~\ref{Lem:M2CondRelax}.

Next, using this result, we can even introduce a more general form of function $M$, namely
	\begin{equation*}
		M(y,\boldxi)=\sum_{i=1}^K k_i(y)M_i(\boldxi)+\tilde M(y,|\boldxi|)\text{ for some }K\in\eN
	\end{equation*}
	where $M_i$, $i=1,\ldots,K$, are spatially independent $\mathcal{N}$--functions and $k_i$ are nonnegative functions. In this case it is sufficient to assume that $\tilde M$ is continuous on $\eR^d\times[0,\infty)$ and satisfies \ref{MTh} while for  functions $k_i$,  we  assume that there exist constants $C_i>1$ such that
$$
\frac{k_i(y_1)}{k_i(y_2)}\leq C_i^{-\frac{1}{\log|y_1-y_2|}} \textrm{ for $i=1,\dots,K$ and any $y_1,y_2\in Y$ with $|y_1-y_2|\leq\frac{1}{2}$.}
$$
Then, keeping the notation from \ref{MTh} and considering an arbitrary $\delta<\delta_0\leq\frac{1}{6\sqrt{d}}$, we have
	\begin{equation*}
	\begin{split}
		M^\delta_j(\boldxi)&=\inf_{\tilde Q^\delta_j}\left(\sum_{i=1}^K k_i(y)M_i(\boldxi)+\tilde M(y,|\boldxi|)\right)\geq \sum_{i=1}^K\inf_{\tilde Q_j^\delta}k_i(y)M_i(\boldxi)+\inf_{\tilde Q_j^\delta}\tilde M(y,|\boldxi|)\\&\geq \sum_{i=1}^K\inf_{\tilde Q^\delta_j}k_i(y)M_i(\boldxi)+(\tilde M^\delta_j)^{**}(|\boldxi|)=: \bar M^\delta_j(\boldxi).
		\end{split}
	\end{equation*}
	Obviously, due to the continuity of functions $k_i$ and $\tilde{M}$ there are points $\bar{y}_i\in \tilde Q_j^\delta$ such that $\bar M^\delta_j(\boldxi)=\sum_{i=1}^K k_i(\bar{y_i})M_i(\boldxi)+(\tilde M^\delta_j)^{**}(|\boldxi|)$. Moreover, the function $\bar M^\delta_j$ is convex with respect to $\boldxi$. Hence we obtain $\bar M^j_\delta(\boldxi)\leq (M^\delta_j)^{**}(\boldxi)$ and since for any $y\in Q^j_\delta$ it follows that $|y-\bar{y}_i|\leq 3\delta\sqrt d$ for all $i=1,\ldots,K$, we get
	\begin{equation*}
		\begin{split}
		\frac{M(y,\boldxi)}{(M^j_\delta)^{**}(\boldxi)}&\leq \sum_{i=1}^K \frac{k_i(y)}{k_i(\bar{y}_i)}+\frac{\tilde M(y,|\boldxi|)}{(\tilde M^\delta_j)^{**}(|\boldxi|)}\leq K\max_{i=1,\ldots,K}\{C_i^{-\frac{1}{\log|y-\bar y_i|}}\}+C\max\{|\boldxi|^{-\frac{D}{\log(E\delta)}}, G^{-\frac{D}{\log(E\delta)}}\}\\
		&\leq \tilde{C}\max\{|\boldxi|^{-\frac{D}{\log(E\delta)}},\max\{\max_{i=1,\ldots,K}\{C_i\},G\}^{-\frac{\max\{D,1\}}{\log(\max\{3\sqrt{d},E\}\delta)}}\}
		\end{split}
	\end{equation*}
	for some constants $\tilde{C},D,E>0$ and $G\geq 1$ if $\delta<\delta_0\leq\min\{\frac{1}{6\sqrt{d}},\frac{1}{2E}\}$ is considered. Notice here, that for the estimate of the second part, we used again the result for radially symmetric functions stated in Lemma~\ref{Lem:M2CondRelax}.

To finish the introduction, we shorty describe here the structure of the paper. In Section~\ref{SS2}, we introduce the function spaces corresponding to our setting, recall several facts about two--scale convergence and most importantly, establish all important properties of the homogenized operator $\hat\boldA$. Then in Section~\ref{SS3} we provide a detailed proof of Theorem~\ref{Thm:Main}. Finally, in Appendix we collected several used tools and results.

\section{Preliminaries}\label{SS2}
Since we deal with rather general function spaces and  growth conditions imposed on the nonlinearity $\boldA$, we recall in  Appendix~\ref{Ape1} the definition of Musielak--Orlicz spaces and also their most important properties. More details about these spaces can be found in \cite{S05,SI69,GGG}. In the forthcoming section we focus only on the specific spaces related to the considered problem. Secondly, in Appendix~\ref{Ape2}, we recall certain technical tools used in the paper. Finally, in Appendix~\ref{Ape3} we recall the existence theorem for the elliptic problems with general growth conditions.

\subsection{Function spaces related to the problem}

In order just to avoid confusion, we remark here that the symbol $L^M$ stands for the Musielak--Orlicz space corresponding to an $\mathcal{N}$--function $M$, while the space $E^M$ denotes the closure of the bounded measurable functions in the topology of $L^M$.
Recall here, that we consider $\Omega\subset \eR^d$ a Lipschitz domain and $Y$ the set $(0,1)^d$. For the $\mathcal{N}$-function $M:Y\times \eR^{d\times N}\to \eR_+$  we use the subscript $y$ to underline  the role of $y$ for the spaces $L^{M_y}(\Omega\times Y;\eR^{d\times N})$ and similarly $E^{M_y}$ endowed with the norm
\begin{equation*}
	\|\boldv\|_{L^{M_y}}=\|\boldv\|_{E^{M_y}}:=\inf\left\{\lambda>0:\int_{\Omega}\int_Y M\left(y,\frac{\boldv(x,y)}{\lambda}\right)\dy \dx\leq 1\right\}.
\end{equation*}
We note that whenever a function dependent on a variable from $Y$ appears, it is always $Y-$periodic although the $Y-$periodicity might not be stressed. We further denote the spaces of smooth periodic or compactly supported functions as
$$
\begin{aligned}
C^\infty_{per}(Y;\eR^{N})&:=\{\boldv\in C^\infty(\Rd; \eR^{N}): \boldv\text{ is } Y\text{-periodic}\},\\
C^\infty_{c}(\Omega;\eR^{N})&:=\{\boldv\in C^\infty(\Rd; \eR^{N}): \textrm{supp }\boldv \textrm{ is compact in } \Omega\}
\end{aligned}
$$
and naturally also the corresponding Bochner spaces $C^\infty_c\left(\Omega;C^\infty_{per}(Y)\right)$. Then the standard Sobolev spaces are defined as
$$
W^{1,1}_0(\Omega;\eR^N):=\overline{\{\boldv\in C^{\infty}_c(\Omega;\eR^N)\}}^{\|\cdot \|_{1,1}}, \qquad W^{1,1}_{per}(Y;\eR^N):=\overline{\{\boldv\in C^\infty_{per}(Y;\eR^N); \, \int_{Y}\boldv =0\}}^{\|\cdot \|_{1,1}}.
$$
Moreover, due to the Poincar\'{e} inequality, we always choose an equivalent norm on $W^{1,1}_0$ and $W^{1,1}_{per}$ as $\|\boldv\|_{1,1}:=\|\nabla \boldv\|_1$. We shall define the Sobolev--Musielak--Orlicz space
\begin{equation*}
W^1_{per}E^M(Y; \eR^N):=\overline{\{\boldv\in C^\infty_{per}(Y;\eR^N); \, \int_{Y}\boldv =0\}}^{\|\cdot \|_{W^1_{per}L^M(Y;\eR^N)}}
\end{equation*}
where $\|\boldv\|_{W^{1}_{per}L^M(Y)}:=\|\nabla \boldv\|_{L^M(Y)}$ and the following  spaces
\begin{align*}
	V^M_0&:=\left\{\boldv\in W^{1,1}_0(\Omega;\eR^N): \nabla \boldv\in L^M(\Omega;\eR^{d\times N})\right\},\\
	V^M_{per}&:=\left\{\boldv\in W^{1,1}_{per}(Y;\eR^N): \nabla \boldv\in L^M(Y;\eR^{d\times N})\right\}.
\end{align*}
In addition, we utilize the following closed subspace of $E^M(Y;\eR^{d\times N})$ and its annihilator
\begin{equation*}
	\begin{split}
		G(Y)&:=\{\nabla \boldw: \boldw\in W_{per}^1E^M(Y;\eR^{N}) \},\\
		G^\bot(Y)&:=\{\boldW^*\in L_{per}^{M^*}(Y;\eR^{d\times N}):\int_Y\boldW^*(y)\cdot\boldW(y)\dy=0\text{ for all }\boldW\in G(Y)\}.
	\end{split}
\end{equation*}
In our situation, the ${\mathcal N}$--function possesses the property of log-H\"older continuity and the following theorem ensures the approximation of every function from $V^M_{per}$ and $V^M_0$ in the sense of modular topology by smooth functions that are periodic or compactly supported, respectively. Below, we use the notation $\ModConvM$ for modular convergence, see Appendix~\ref{Ape1}.

\begin{Lemma}\label{Thm:ModDensity}
	Let $\Sigma\subset\Rd$ be a bounded domain and an ${\mathcal N}$--function $M$ satisfy \ref{MTwo} and \ref{M4} with $\Sigma$ replacing $Y$.
Then we have the following modular convergence results:
	\begin{enumerate}[label=\arabic*)]
		\item\label{MDO} Let $\Sigma$ be Lipschitz. Then for  any scalar function $v\in V_0^M\cap L^\infty(\Sigma)$ there exists a sequence $\{v^k\}_{k=1}^\infty\subset C^\infty_c(\Sigma)$ such that $\nabla v^k\ModConvM\nabla v$.
		\item\label{MDTw} Let $\Sigma$ be star--shaped. Then for any function $\boldv\in V_0^M$ there exists a sequence $\{\boldv^k\}_{k=1}^\infty\subset C^\infty_c\left(\Sigma;\eR^N\right)$ such that $\nabla \boldv^k\ModConvM\nabla \boldv$.
		\item\label{MDTh} Let $\Sigma=Y$. Then for any function $\boldv\in V_{per}^M$ there exists a sequence $\{\boldv^k\}_{k=1}^\infty\subset C^\infty_{per}\left(\Sigma;\eR^N\right)$ such that $\nabla \boldv^k\ModConvM\nabla \boldv$.
\item\label{MDOMB} Let $\Sigma$ be Lipschitz and the embedding $W^{1,m_1}(\Sigma) \hookrightarrow L^{m_2}(\Sigma)$ hold. Then for  any  function $v\in V_0^M$ there exists a sequence $\{v^k\}_{k=1}^\infty\subset C^\infty_c(\Sigma)$ such that $\nabla v^k\ModConvM\nabla v$.
	\end{enumerate}
	\begin{proof}
		The assertion~\ref{MDO} is covered by \cite[Theorem 2.2]{GSZG17}. To prove the assertions \ref{MDTw}--\ref{MDOMB} one follows the common scheme:
		\begin{enumerate}
			\item Construction of the mollification $\nabla \boldv^{\delta_k}$ of $\nabla \boldv$.
			\item Showing that the family $\{\nabla \boldv^{\delta_k}\}_{k=1}^\infty$ is uniformly bounded in $L^M(\Sigma;\eR^{d\times N})$.
			\item Showing that $\nabla \boldv^{\delta_k}\ModConvM\nabla \boldv$.
		\end{enumerate}
		The detailed proof can be performed by repeating Steps 1-3 from the proof of \cite[Theorem 2.2]{GSZG17}.
	\end{proof}
\end{Lemma}

We state several technical lemmas.
\begin{Lemma}{\cite[Lemma 2.1.]{GSG08}}\label{Lem:MConvEquiv}
	Let $N\geq 1$, $M$ be an $\mathcal{N}$--function and $\{\boldv^k\}_{k=1}^\infty$ be a sequence of measurable $\eR^N-$valued functions on $\Sigma$. Then $\boldv^k\ModConvM\boldv$ in $L^M(\Sigma;\eR^N)$ if and only if $\boldv^k\rightarrow\boldv$ in measure and there exists some $\lambda>0$ such that $\{M(\cdot,\lambda \boldv^k)\}_{k=1}^\infty$ is uniformly integrable, i.e.,
	\begin{equation*}
		\lim_{R\rightarrow\infty}\left(\sup_{k\in\eN}\int_{\{x:|M(x,\lambda\boldv^k(x))|>R\}}M(x,\lambda\boldv^k(x))\dx\right)=0.
	\end{equation*}
\end{Lemma}
\begin{Lemma}\cite[Lemma 2.2.]{GSG08}\label{Lem:UnifIntegr}
Let $M$ be an ${\mathcal N}$--function and assume that there is $c>0$ such that $\int_\Sigma M(x,\boldv^k)\dx\leq c$ for all $k\in\eN$. Then $\{\boldv^k\}_{k=1}^\infty$ is uniformly integrable.
\end{Lemma}
\begin{Lemma}\label{Lem:ModConvGrTrunc}
	Let $M$ be an ${\mathcal N}$--function and $\Sigma$ be a bounded domain. Then for any  $v\in V_0^M(\Sigma)$ we have $\nabla T_k(v)\ModConvM\nabla v$ as $k\rightarrow\infty$, where
	\begin{equation*}
		T_k (v)=\begin{cases}
			v &\text{ if }|v|\leq k\\
			k\frac{v}{|v|} &\text{ if }|v|> k.
		\end{cases}
	\end{equation*}
	\begin{proof}
	Clearly, $\nabla T_k(u)\rightarrow \nabla u$ a.e. in $\Sigma$, which has finite measure. Hence the sequence $\{\nabla T_k(u)\}_{k=1}^\infty$ converges to $\nabla u$ in measure. Moreover, as $M(\cdot,\nabla T_k(u))\leq M(\cdot,\nabla u)$ a.e. in $\Sigma$ by the definition of $T_k$, Lemma~\ref{Lem:UnifIntegr} implies that $\{\nabla T_k(u)\}_{k=1}^\infty$ is uniformly integrable. These two facts are equivalent to $\nabla T_k(u)\ModConvM\nabla u$ according to Lemma~\ref{Lem:MConvEquiv}.
	\end{proof}
\end{Lemma}

\subsection{Standard tools used  for homogenization}

This section is devoted to the introduction of the two--scale convergence via periodic unfolding. This approach allows to represent the weak two--scale convergence by means of the standard weak convergence in a Lebesgue space on the product $\Omega\times Y$, details for the case of $L^p$ spaces can be found in \cite{V06}. In the same manner the strong two--scale convergence is introduced. Since function spaces, which we are working with, provide only the weak$^*$ compactness of bounded sets, we introduce the two--scale compactness in the weak$^*$ sense. However, it turns out that this notion of convergence and some of its properties are sufficient for our purposes. We define functions $n:\eR\rightarrow \eZ$ and $N:\Rd\rightarrow \eZ^d$ as
	\begin{equation*}
			n(t)=\max\{n\in \eZ: n\leq t\}\ \forall t\in\eR,	N(x)=(n(x_1),\ldots, n(x_d))\ \forall x\in\Rd.
	\end{equation*}
	
	Then we have for any $x\in \Rd,\varepsilon>0$, a two--scale decomposition $x=\varepsilon\left(N\left(\frac{x}{\varepsilon}\right)+R\left(\frac{x}{\varepsilon}\right)\right)$. We also define for any $\varepsilon>0$ a two--scale composition function $S_\varepsilon:\Rd\times Y\rightarrow\Rd$ as $S_\varepsilon(x,y):=\varepsilon \left(N\left(\frac{x}{\varepsilon}\right)+y\right)$.
	It follows immediately that
	\begin{equation}\label{TSCompUnifConv}
		S_\varepsilon(x,y)\rightarrow x\text{ uniformly in }\Rd\times Y\text{ as }\varepsilon\rightarrow 0
	\end{equation}
	since $S_\varepsilon(x,y)=x+\varepsilon\left(y-R\left(\frac{x}{\varepsilon}\right)\right)$. In the rest of the section we assume that $m:[0,\infty) \to [0,\infty)$ is an ${\mathcal N}$--function.\\
We say that a sequence of functions $\{v^\varepsilon\}\subset L^m(\Rd)$
\begin{enumerate}
	\item converges to $v^0$ weakly$^*$ two--scale in $L^m(\Rd\times Y)$, $v^\varepsilon\WTSSCon v^0$, if $v^\varepsilon\circ S_\varepsilon$ converges to $v^0$ weakly$^*$ in $L^m(\Rd\times Y)$,
	\item converges to $v^0$ strongly two--scale in $E^m(\Rd\times Y)$, $v^\varepsilon\STSCon v^0$, if $v^\varepsilon\circ S_\varepsilon$ converges to $v^0$ strongly in $E^m(\Rd\times Y)$.
\end{enumerate}
We define two--scale convergence in $L^m(\Omega\times Y)$ as two--scale convergence in $L^m(\Rd\times Y)$ for functions extended by zero to $\Rd\setminus\Omega$. The following lemma will be utilized to express properties of two--scale convergence in terms of single-scale convergence.
\begin{Lemma}\cite[Lemma 1.1]{V06}\label{Lem:Decomp} Let $g$ be measurable with respect to a $\sigma-$algebra generated by the product of the $\sigma-$algebra of all Lebesgue--measurable subsets of $\eR^d$ and the $\sigma-$algebra of all Borel--measurable subsets of $Y$. Assume in addition that $g\in L^1(\eR^d;L^\infty_{per}(Y))$ and extend it by $Y-$periodicity to $\eR^d$ for a.a. $x\in\eR^d$. Then, for any $\varepsilon>0$, the function $(x,y)\mapsto g(S_\varepsilon(x,y),y)$ is integrable and
	\begin{equation*}
		\int_{\Rd}g\left(x,\frac{x}{\varepsilon}\right)\dx=\int_{\Rd}\int_Y g(S_\varepsilon(x,y),y)\dy\dx.
	\end{equation*}
\end{Lemma}
Several useful properties of the two--scale convergence are summarized in the following lemma.
\begin{Lemma}\label{Lem:Facts2S}
Assume that $m:[0,\infty)\rightarrow[0,\infty)$ is an ${\mathcal N}$--function.
\begin{enumerate}[label=(\roman*)]
	\item\label{F2SFir} Let $v:\Omega\times Y\rightarrow\eR$ be  Carath\'eodory, $v\in E^m(\Omega\times Y)$, $v$ be $Y-$periodic, define $v^\varepsilon(x)=v(x,\frac x\varepsilon)$ for $x\in\Omega$. Then $v^\varepsilon\STSCon v$ in $E^m(\Omega\times Y)$ as $\varepsilon\rightarrow 0$.
	\item\label{F2SS} Let $v^\varepsilon\WTSSCon v^0$ in $L^m(\Omega\times Y)$ then $v^\varepsilon\rightharpoonup^* \int_Y v^0(\cdot,y)\dy$ in $L^m(\Omega)$.
		\item\label{F2ST} Let $v^{\varepsilon}\WTSSCon v^0$ in $L^{m}(\Omega\times Y)$ and $w^{\varepsilon}\STSCon w^0$ in $E^{m^*}(\Omega\times Y)$ then $\int_\Omega v^{\varepsilon} w^{\varepsilon}\to\int_\Omega\int_Yv^0w^0$.
	\item\label{F2SFo} Let $v^{\varepsilon}\WTSSCon v^0$ in $L^m(\Omega\times Y)$ then for any $\psi\in C^\infty_c\left(\Omega;C^\infty_{per}(Y)\right)$
	\begin{equation*}
		\lim_{\varepsilon\rightarrow 0}\int_\Omega v^{\varepsilon}(x) \psi\left(x,\frac{x}{\varepsilon}\right)\dx=\int_\Omega\int_Y v^0(x,y)\psi(x,y)\dy\dx.
	\end{equation*}
	\item\label{F2SFi} Let $\{v^\varepsilon\}$ be a bounded sequence in $L^m(\Omega)$. Then there is $v^0\in L^m(\Omega\times Y)$ and a subsequence $\varepsilon_k\to0$ as $k\to\infty$ such that $v^{\varepsilon_k}\WTSSCon v^0$ in $L^m(\Omega\times Y)$ as $k\to\infty$.
	\item\label{F2SSi} Let $\{v^{\varepsilon}\}\subset V_0^m$ be such that
	\begin{equation}
		\begin{alignedat}{2}
			v^\varepsilon&\WSCon v&&\text{ in }L^m(\Omega),\\
			\nabla v^\varepsilon&\WSCon \nabla v&&\text{ in }L^m(\Omega;\eR^d).
			\end{alignedat}
	\end{equation}
	Then $v^{\varepsilon}\WTSSCon v$ in $L^m(\Omega\times Y)$ and there is a subsequence $\varepsilon_k\rightarrow 0$ as $k\rightarrow \infty$ and $\mathbf{v}\in L^m(\Omega\times Y;\Rd)$ such that $\nabla v^{\varepsilon_k}\WTSSCon\nabla v+\boldv$ in $L^m(\Omega\times Y;\Rd)$ as $k\to\infty$ and for a.a. $x\in\Omega$ and any $\bpsi\in C^\infty_{per}(Y;\Rd)$ fulfilling $\div \bpsi=0$ in $Y$, there holds
	\begin{equation*}
		\int_Y \mathbf{v}(x,y)\cdot \bpsi(y)\dy=0.
	\end{equation*}
\item\label{F2SSev}  Let $\Phi:\Rd\times \eR^{d\times N} \rightarrow \eR$ satisfy:
\begin{enumerate}[label=(\alph*)]
	\item $\Phi$ is Carath\'eodory,
	\item $\Phi(\cdot,\boldxi)$ is $Y-$periodic for any $\boldxi\in\eR^{d\times N}$, $\Phi(y,\cdot)$ is convex for almost all $y\in Y$,
	\item $\Phi\geq 0$, $\Phi(\cdot,0)=0$.
\end{enumerate}
Then for any sequence $\boldUeps\WTSSCon \boldU$ in $L^m(\Omega\times Y;\eR^{d\times N})$ it follows that
\begin{equation*}
\liminf_{\varepsilon\rightarrow \infty}\int_{\Omega}\Phi\left(\frac{x}{\varepsilon},\boldUeps(x)\right)\dx\geq \int_{\Omega\times Y}\Phi(y,\boldU(x,y))\dy\dx.
\end{equation*}
\end{enumerate}
\begin{proof}
By Lemma~\ref{Lem:Decomp} we have for $v$ extended by zero on $(\Rd\setminus\Omega)\times Y$ that 
\begin{equation*}
	v^\varepsilon\circ S_\varepsilon (x,y)=v\left(S_\varepsilon(x,y),\frac{S_\varepsilon(x,y)}{\varepsilon}\right)=v\left(\varepsilon \left(N\left(\frac{x}{\varepsilon}\right)+y\right),y\right)
\end{equation*}
is an integrable function of $(x,y)$. According to \cite[Theorem 3.15.5]{KJF77} $v\in E^{m}(\Omega\times Y)$ is $m-$mean continuous, i.e., for given $\eta>0$ there exists $\kappa>0$ such that
$\|v_h-v\|_{L^m}\leq \eta$ for $h=(h^1,h^2)\in\eR^{2d}$ with $|h|<\kappa$, where
\begin{equation*}
	v_h(x,y):=\begin{cases}
		v(x+h^1,y+h^2)&\text{ if }(x+h^1,y+h^2)\in \Omega\times Y,\\
		0&\text{otherwise}.
	\end{cases}
\end{equation*}
Hence for fixed $\eta>0$ we find $\kappa>0$ such that $\|v_h-v\|_{L^m(\Omega\times Y)}<\eta$ for all $|h|<\kappa$. Due to~\eqref{TSCompUnifConv} we find $\varepsilon_0>0$ such that for all $\varepsilon\leq\varepsilon_0$ $\left\|\varepsilon(N(\frac{x}{\varepsilon})+y)-x\right\|_{L^\infty(\Rd\times Y)}<\kappa$. For fixed $\eta$ we found $\varepsilon_0>0$ such that for all $\varepsilon\leq\varepsilon_0$ we have $\|v^\varepsilon\circ S_\varepsilon-v\|_{L^m(\Omega\times Y)}<\eta$, which concludes~\ref{F2SFir}.

We obtain~\ref{F2SS} once we use in the definition of the weak$^*$ two--scale convergence in $L^m(\Omega\times Y)$ test functions, which are  independent of $y$-variable.

Assertion~\ref{F2ST} follows immediately from the definition of the weak$^*$ two--scale convergence in $L^m(\Omega)$, strong two--scale convergence in $E^{m^*}(\Omega)$ and Lemma~\ref{Lem:Decomp} applied to the function $g=v^{\varepsilon} w^{\varepsilon}$ independent of $y$.

To show assertion~\ref{F2SFo} we fix a weakly$^*$ two--scale convergent sequence $\{v^{\varepsilon_k}\}_{k=1}^\infty\subset L^m(\Omega)$ with a limit $v^0\in L^m(\Omega\times Y)$ and $\psi\in C^\infty_c(\Omega;C^\infty_{per}(Y))$. Then we have $v^{\varepsilon_k}(x)\psi(x,y)\in L^1(\Rd;L^\infty_{per}(Y))$ provided that we set $v^{\varepsilon_k}=0$ in $\eR^d\setminus\Omega$, $\psi=0$ in $(\eR^d\setminus\Omega)\times Y$. Therefore by Lemma~\ref{Lem:Decomp} we get
\begin{equation*}
	\int_\Omega v^{\varepsilon_k}(x)\psi\left(x,\frac{x}{\varepsilon_k}\right)\dx=\int_{\Omega}\int_Y v^{\varepsilon_k}(S_{\varepsilon_k}(x,y))\psi(S_{\varepsilon_k}(x,y),y)\dy\dx.
\end{equation*}
Combining this with the convergence results $v^{\varepsilon_k}\WTSSCon v^0$ in $L^m(\Omega\times Y)$ and $\psi\left(x,\frac{x}{\varepsilon_k}\right)\STSCon\psi(x,y)$ in $E^{m^*}(\Omega\times Y)$ as $k\rightarrow\infty$, which follows by assertion~\ref{F2SFir}, we infer
\begin{equation*}
\lim_{k\rightarrow\infty}\int_{\Omega}v^{\varepsilon_k}(x)\psi\left(x,\frac{x}{\varepsilon_k}\right)\dx
=\lim_{k\rightarrow}\int_{\Omega}\int_Yv^{\varepsilon_k}(S_{\varepsilon_k}(x,y))\psi\left(S_{\varepsilon_k}(x,y),y\right)\dy\dx=\int_\Omega\int_Y v^0(x,y)\psi(x,y)\dy\dx
\end{equation*}
by assertion~\ref{F2ST}.\\
In order to show~\ref{F2SFi}, we first realize that for any $\{v^\varepsilon\}$ bounded in $L^m(\Omega)$ Lemma~\ref{Lem:Decomp} applied to a function $g=m\left(\frac{|v^\varepsilon|}{\lambda}\right)$ independent of $y$ implies
\begin{equation*}
c\geq \int_\Omega m\left(\frac{|v^{\varepsilon}(x)|}{\lambda}\right)\dx=\int_\Omega\int_Y m\left(\frac{|v^{\varepsilon}(S_\varepsilon(x,y))|}{\lambda}\right)\dy\dx
\end{equation*}
for some $\lambda>0$.  We deduce the existence of a selected subsequence $\{v^{\varepsilon_k}\circ S_{\varepsilon_k}\}\subset \{v^{\varepsilon}\circ S_\varepsilon\}$ and the limit function $v^0\in L^m(\Omega\times Y)$ such that $v^{\varepsilon_k}\circ S_{\varepsilon_k}\WSCon v^0$ in $L^m(\Omega\times Y)$ as $k\rightarrow 0$ by the Banach-Alaoglu theorem for spaces with a separable predual. We recall that $L^m(\Omega\times Y)=\left(E^{m^*}(\Omega\times Y)\right)^*$. Assertion~\ref{F2SFi} obviously follows by the definition of weak$^*$ two--scale convergence.

In order to show~\ref{F2SSi} we observe first that $\{v^\varepsilon\}_{\varepsilon\in(0,1)}$ is bounded in $L^m(\Omega)$. Thus by~\ref{F2SFi} there is a sequence $\varepsilon_k\rightarrow 0$ as $k\rightarrow \infty$ and $v^0\in L^m(\Omega\times Y)$ such that $v^{\varepsilon_k}\WTSSCon v^0$ in $L^m(\Omega\times Y)$. Then~\ref{F2ST} implies for all $\varphi\in C^\infty_c(\Omega,C^\infty_{per}(Y)^d)$ that
\begin{equation*}
	\begin{split}
	0&=-\lim_{k\to \infty}\varepsilon_k\int_{\Omega}\nabla v^{\varepsilon_k}(x)\cdot \left[\varphi\left(x,\frac{x}{\varepsilon_k}\right)\right]\dx=\lim_{k\to \infty}\varepsilon_k\int_{\Omega}v^{\varepsilon_k}(x)\div\left[\varphi\left(x,\frac{x}{\varepsilon_k}\right)\right]\dx\\
&=\lim_{k\to\infty}\int_{\Omega}\varepsilon_k v^{\varepsilon_k}(x)\div_x\varphi\left(x,\frac{x}{\varepsilon_k}\right)+v^{\varepsilon_k}(x)\div_y\varphi\left(x,\frac{x}{\varepsilon_k}\right)\dx
=\int_{\Omega}\int_Y v^0(x,y)\div_y \varphi(x,y)\dy\dx,
	\end{split}
\end{equation*}
which implies that $v^0$ is independent of $y$. As $v=\int_Y v^0$ by~\ref{F2SS}, we see that for any weakly$^*$ two--scale convergent subsequence of $\{v^\varepsilon\}$ the limit is $v$. Hence $v$ is the weak$^*$ two--scale limit of the entire sequence $\{v^\varepsilon\}$. Applying~\ref{F2SFo} on the sequence $\{\nabla v^{\varepsilon_k}\}$ we get the subsequence $\{v^{\varepsilon_k}\}$ (that will not be relabeled) and $\boldw\in L^m(\Omega\times Y;\Rd)$ such that $\nabla v^{\varepsilon_k}\WTSSCon\boldw$ in $L^m(\Omega\times Y;\Rd)$ as $k\to\infty$. Let us choose $z\in C^\infty_c(\Omega)$ and $\bpsi\in C^\infty_{per}(Y;\Rd)$ with $\div_y \bpsi=0$ in $Y$. Then it follows from (i) (applied to $\bpsi$) and (iii) that
\begin{equation*}
	\lim_{k\to\infty}\int_{\Omega}\nabla v^{\varepsilon_k}(x)\cdot z(x)\bpsi\left(\frac{x}{\varepsilon_k}\right)\dx=\int_\Omega\int_Y \boldw(x,y)\cdot z(x)\bpsi(y)\dy\dx
\end{equation*}
whereas the integration by parts yields
\begin{equation*}
	\begin{split}
	&\lim_{k\to\infty}\int_{\Omega}\nabla v^{\varepsilon_k}(x)\cdot z(x)\bpsi\left(\frac{x}{\varepsilon_k}\right)\dx=-\lim_{k\to\infty}\int_{\Omega} v^{\varepsilon_k}(x)\nabla z(x)\cdot \bpsi\left(\frac{x}{\varepsilon_k}\right)\dx\\&=-\int_{\Omega}\int_Y v(x)\nabla z(x)\cdot\bpsi(y)\dy\dx=\int_{\Omega}\int_Y \nabla v(x)\cdot z(x)\bpsi(y)\dy\dx.
	\end{split}
\end{equation*}
Hence the function $\boldv=\boldw-\nabla v$ has all required properties. \\
Let us show~\ref{F2SSev}. It follows from Lemma~\ref{Lem:Decomp} and Lemma~\ref{Lem:SemCon} that for $\boldUeps,\boldU$ extended by zero in $\Rd\setminus\Omega$
\begin{equation*}
	\liminf_{\varepsilon\rightarrow 0}\int_\Omega \Phi\left(\frac{x}{\varepsilon},\boldUeps(x)\right)\dx=\liminf_{\varepsilon\rightarrow 0}\int_{\Omega\times Y}\Phi\left(y,\boldUeps(S_{\varepsilon}(x,y)\right)\dx\dy\geq \int_{\Omega\times Y}\Phi(y,\boldU(x,y))\dx\dy
\end{equation*}
since $\boldUeps\WTSSCon\boldU$ in $L^m(\Omega\times Y; \eR^{d\times N})$ implies $\boldUeps\WCon\boldU$ in $L^1(\Omega\times Y; \eR^{d\times N})$. Hence we conclude~\ref{F2SSev}.
\end{proof}
\end{Lemma}
\subsection{Properties of the mapping $\hat{\boldA}$}
Let us define an operator $\hat{\boldA}\!:\eR^{d\times N}\rightarrow\eR^{d\times N}$ as
\begin{equation}\label{AhDef}
	\hat{\boldA}(\boldxi)=\int_Y\boldA(y,\boldxi+ \nabla \boldw_{\boldxi})\dy
\end{equation}
where the $Y-$periodic function $\boldw_{\boldxi}$ is a unique solution of the following cell problem
\begin{equation}\label{CellPr}
	\div\boldA(y,\boldxi+\nabla \boldw_{\boldxi})=0\text{ in }Y.
\end{equation}
In what follows, we show that this definition is meaningful and derive the essential properties of the operator $\hat{\boldA}$ needed later for the homogenization problem.
\begin{Lemma}\label{Lem:cellex}
	Let $Y=(0,1)^d$, the operator $\boldA$ satisfy \ref{AO}--\ref{AF} and the ${\mathcal N}$--function $M$ satisfy \ref{MO}--\ref{M4}. Then the problem~\eqref{CellPr} admits a unique weak solution $\boldw_{\boldxi}\in V^M_{per}$ satisfying for all $\bphi\in V^M_{per}$
	\begin{equation}\label{CellPrWF}
	\int_{Y} \boldA\left(y,\boldxi+\nabla \boldw_{\boldxi}(y)\right)\cdot\nabla\bphi(y)\dy=0.
\end{equation}
Moreover,
\begin{equation}\label{WSCont}
	\boldxi^k\rightarrow \boldxi\text{ in }\eR^{d\times N}\text{ implies }\boldA(\cdot,\boldxi^k+\nabla \boldw^k)\WSCon\boldA(\cdot,\boldxi+\nabla \boldw)\text{ in }L^{M^*}(Y;\eR^{d\times N}),
\end{equation}
where $\boldw^k$ is a solution of the cell problem corresponding to $\boldxi^k$ and $\boldw$ to $\boldxi$.
	\begin{proof}
		We omit existence and uniqueness proofs since it suffices to modify straightforwardly the methods used in the proofs of Theorem~\ref{Thm:WSExistMHoeldC} in the appendix. Notice here that we do not have any restriction on the geometry since we deal only with spatially periodic setting.
		
		Let us assume that $\{\boldxi^k\}_{k=1}^\infty$ is such that $\boldxi^k\rightarrow\tilde\boldxi$ in $\eR^{d\times N}$ as $k\rightarrow\infty$. We denote by $\boldw^k$ the solution of the cell problem corresponding to $\boldxi^k$ and by $\tilde\boldw$ the solution corresponding to $\tilde\boldxi$. We also denote $\boldZ^k(y):=\boldA(y,\boldxi^k+\nabla \boldw^k(y))$. First, we show that
		\begin{equation}\label{ZNEst}
			\int_YM(y,\boldxi^k+\nabla \boldw^k(y))+M^*(y,\boldZ^k(y))\dy\leq c.
		\end{equation}
		Since $\boldw_k$ is always an admissible test function in~\eqref{CellPrWF} for $\boldxi:=\boldxi^k$,  we directly obtain
		\begin{equation}\label{TestCPrSol}
			\int_Y \boldZ^k(y)\cdot\nabla \boldw^k(y)\dy=0.
		\end{equation}
Hence, using~\ref{ATh}, \eqref{TestCPrSol} and the Young inequality yields	(assuming without loss of generality that $c\le 1$)
		\begin{equation*}
			\begin{split}
			c&\int_Y M^*(y,\boldZ^k(y))+M(y,\boldxi^k+\nabla \boldw^k)\dy\leq \int_Y \boldZ^k\cdot(\boldxi^k+\nabla \boldw^k)\dy=\int_Y \boldZ^k\cdot\boldxi^k\dy\\
			&\leq \frac{c}{2}\int_Y M^*(y,\boldZ^k)\dy+\int_Y M\left(y,\frac{2}{c}\boldxi^k\right)\dy.
			\end{split}
		\end{equation*}
The second integral on the right hand side is finite due to~\ref{MTwo} as $\{\boldxi^k\}_{k=1}^\infty$ is bounded.	Without loss of generality, we can assume that
		\begin{equation}\label{WSConWNZN}
			\begin{alignedat}{2}
				\nabla\boldw^k&\WSCon\nabla\bar \boldw&&\text{ in }L^M(Y;\eR^{N}),\\
				\boldZ^k&\WSCon\boldZ&&\text{ in }L^{M^*}_{per}(Y;\eR^{d\times N})
			\end{alignedat}
		\end{equation}
		as $k\rightarrow\infty$. We show that $\bar \boldw=\tilde\boldw$ and $\boldZ=\boldA(\cdot,\tilde\boldxi+\nabla \tilde\boldw)$. We immediately obtain that
		\begin{equation}\label{ZXiCon}
			\lim_{k\rightarrow\infty}\int_Y \boldZ^k(y)\cdot\boldxi^k\dy=\int_Y \boldZ(y)\cdot\tilde\boldxi\dy.
		\end{equation}
Further, we also use the following identity
\begin{equation}\label{EqForZ}
	\int_Y \boldZ(y)\cdot\nabla\bphi(y)\dy=0 
\end{equation}
for all $\bphi\in V_{per}^M$. In order to show it, we observe that from~\eqref{WSConWNZN} and the definition of $\boldZ^k$ the identity~\eqref{EqForZ} follows for all $\bphi \in W^1_{per}E^M(Y; \eR^d)$. Since $M$ satisfies~\ref{MTh}, we can use the density of smooth functions in the modular topology, see Step 5 of Theorem~\ref{Thm:WSExistMHoeldC}, to deduce~\eqref{EqForZ} for all $\bphi\in V_{per}^M$.
From~\eqref{TestCPrSol}, \eqref{ZXiCon} and~\eqref{EqForZ} we infer
\begin{equation}\label{IdentLimBZ}
	\lim_{k\rightarrow\infty}\int_Y \boldZ^k(y)\cdot(\boldxi^k+\nabla \boldw^k(y))\dy=\int_Y\boldZ(y)\cdot(\tilde\boldxi+\nabla \bar \boldw(y))\dy.
\end{equation}
Since $\boldA(x,0)=0$ and $\boldA$ is monotone, the negative part of $\boldZ^k\cdot(\boldxi^k+\nabla \boldw^k)$ is trivially weakly compact in $L^1(Y)$. Due to Lemma~\ref{Lem:YMProp} and~\eqref{IdentLimBZ} we get
		\begin{equation}\label{EssIneqYM}
			\int_Y\int_{\eR^{d\times N}} \boldA(y,\boldzeta)\cdot\boldzeta\d\nu_y(\boldzeta)\dy\leq\liminf_{k\rightarrow\infty}\int_Y \boldZ^k(y)\cdot(\boldxi^k+\nabla \boldw^k(y))\dy= \int_\Omega\boldZ(y)\cdot(\tilde\boldxi+\nabla \bar\boldw(y))\dy,
		\end{equation}
		where $\nu_y$ is the Young measure generated by $\{\boldxi^k+\nabla \boldw^k\}_{k=1}^\infty$. The monotonicity of $\boldA$ yields
		\begin{equation}\label{MonY}
			\int_Y\int_{\eR^{d\times N}}h(y,\boldzeta)\d\nu_y(\boldzeta)\dy\geq 0
		\end{equation}
		for $h(y,\boldzeta):=(\boldA(y,\boldzeta)-\boldA(y,\tilde\boldxi+\nabla \bar \boldw))\cdot(\boldzeta-\tilde\boldxi-\nabla \bar \boldw)$. Since $\{\boldxi^k+\nabla \boldw^k\}_{k=1}^\infty$ and $\{\boldA(\cdot,\boldxi^k+\nabla \boldw^k)\}_{k=1}^\infty$ are weakly relatively compact due to~\eqref{ZNEst} and $\boldA$ is a Carath\'eodory function, Lemma~\ref{Lem:YMProp} implies
		\begin{equation}\label{YMLimIdent}
			\begin{alignedat}{2}
				\tilde\boldxi+\nabla \bar \boldw&=\int_{\eR^{d\times N}}\boldzeta\d\nu_y(\boldzeta)&&\text{ a.e. in }Y,\\
				\boldZ&=\int_{\eR^{d\times N}}\boldA(\cdot,\boldzeta)\d\nu_y(\boldzeta)&&\text{ a.e. in }Y.
			\end{alignedat}
		\end{equation}
		Then we get
		\begin{equation*}
			\int_Y\int_{\eR^{d\times N}}h(y,\boldzeta)\d\nu_y(\boldzeta)\dy=\int_Y \int_{\eR^{d\times N}}\boldA(y,\boldzeta)\cdot\boldzeta\d\nu_y(\boldzeta)\dy-\int_Y \boldZ\cdot(\tilde\boldxi+\nabla \bar \boldw)\leq  0
		\end{equation*}
		by~\eqref{EssIneqYM}. Combining this with~\eqref{MonY} we obtain $\int_{\eR^{d\times N}} h(y,\boldzeta)\d\nu_y(\boldzeta)=0$ for a.a. $y\in Y$. As $\nu_y$ is a probability measure and $\boldA$ is strictly monotone, we infer that $\supp\{\nu_y\}=\{\tilde\boldxi+\nabla \bar \boldw\}$ a.e. in $Y$. Thus we have $\nu_y=\delta_{\tilde\boldxi+\nabla \bar \boldw(y)}$ a.e. in $Y$. Inserting this into~\eqref{YMLimIdent}$_2$ yields $\boldZ(y)=\boldA(y,\tilde\boldxi+\nabla \bar \boldw(y))$. Hence we infer due to~\eqref{EqForZ} that $\bar \boldw$ is a weak solution to~\eqref{CellPr} corresponding to $\tilde\boldxi$. Since this solution is unique, we obtain $\bar \boldw=\tilde\boldw$. Up to now we have shown that from $\{\boldZ^k\}_{k=1}^\infty$ there can be extracted a subsequence that converges weakly$^*$ to $\boldA(\cdot,\tilde\boldxi+\nabla\tilde\boldw)$ in $L^{M^*}(Y;\eR^{d\times N})$. The uniqueness of this limit implies that the whole sequence $\{\boldZ^k\}_{k=1}^{\infty}$ must converge to $\boldA(\cdot,\tilde\boldxi+\nabla\tilde\boldw)$, which finishes the proof.
	\end{proof}
\end{Lemma}

Now, we investigate the properties of a functional $f:\eR^{d\times N}\rightarrow [0,\infty)$ defined as
\begin{equation}
	f(\boldxi)=\inf_{\boldW\in G(Y)}\int_YM(y,\boldxi+\boldW(y))\dy. \label{DEF:f}
\end{equation}

\begin{Lemma}\label{Lem:FProp}
	Let ${\mathcal N}$--function $M$ satisfy \ref{MO}--\ref{MTwo}.  Then the functional $f$ defined in~\eqref{DEF:f} is an ${\mathcal N}$--function, i.e., it satisfies:
	\begin{enumerate}[label=\arabic*)]
		\item \label{fFi} $f(\boldxi)=0$ if and only if $\boldxi=\bzero$,
		\item \label{fS} $f(\boldxi)=f(-\boldxi)$,
		\item  $f$ is convex,
		\item \label{fFo} $\lim_{|\boldxi|\rightarrow 0}\frac{f(\boldxi)}{|\boldxi|}=0$, $\lim_{|\boldxi|\rightarrow\infty}\frac{f(\boldxi)}{|\boldxi|}=\infty$.
	\end{enumerate}
	\begin{proof}
		First, we show that
		\begin{equation}\label{AbBeEstF}
			m_1(|\boldxi|)\leq f(\boldxi)\leq m_2(|\boldxi|).
		\end{equation}
		Let us show the first inequality in the latter estimate. Using~\ref{MTwo}, Jensen's inequality and the fact that the average over $Y$ of the gradient of an $Y-$periodic function vanishes we have
		\begin{equation*}
		\begin{split}
			f(\boldxi)&=\inf_{\boldW\in G(Y)}\int_Y M(y,\boldxi+\boldW(y))\dy\geq \inf_{\boldW\in G(Y)}\int_Y m_1(|\boldxi+\boldW(y)|)\dy\geq \inf_{\boldW\in G(Y)}m_1\left(\left|\boldxi+\int_Y\boldW(y)\dy\right|\right)\\
			&\geq m_1(|\boldxi|).
			\end{split}
		\end{equation*}
		On the other hand we get by~\ref{MTwo} that $f(\boldxi)\leq m_2(|\boldxi|)$ since $\bzero\in G(Y)$, which follows from the fact that $G$ is a subspace of $E^M(Y;\eR^{d\times N})$.\\
		Assertions~\ref{fFi} and~\ref{fFo} then follow immediately from~\eqref{AbBeEstF}.\\
		Obviously, since $M$ is even in the second argument and $G(Y)$ is a subspace of $E^M_{per}(Y;\eR^{d\times N})$ we have~\ref{fS}.\\
	In order to show the convexity of $f$ we take $\lambda\in (0,1)$, $\boldxi_1,\boldxi_2\in\eR^{d\times N}$ and $\boldW_1,\boldW_2\in G(Y)$. Again the fact that $G(Y)$ is a subspace of $E^M_{per}(Y;\eR^{d\times N})$ and the convexity of $M$ yields
		\begin{equation*}
				f(\lambda\boldxi_1+(1-\lambda)\boldxi_2)\leq \lambda\int_Y  M(y,\boldxi_1+ \boldW_1(y))\dy+(1-\lambda)\int_Y M(y,\boldxi_2+\boldW_2(y))\dy.
		\end{equation*}
		One obtains the desired conclusion by taking the infimum over $\boldW_1$ and $\boldW_2$ on the right hand side of the latter inequality.
	\end{proof}	
\end{Lemma}

\begin{Lemma}\label{Lem:FSt}
	Let ${\mathcal N}$--function $M$ satisfy \ref{MO}--\ref{MTwo} and $f$ be defined by~\eqref{DEF:f}. Then the conjugate ${\mathcal N}$--function $f^*$ to $f$ is given by
	\begin{equation}\label{FSt}
		f^*(\boldxi)=\inf_{\substack{\boldW^*\in G^\bot(Y),\\ \int_Y \boldW^*(y)\dy=\boldxi}}\int_Y M^*(y,\boldW^*(y))\dy.
	\end{equation}
	\begin{proof}
		 Using the fact that the average over $Y$ of a gradient of $Y-$periodic function vanishes we obtain defining a functional $\mathcal{F}:L^M(Y;\eR^{d\times N})\rightarrow\eR$ as
		\begin{equation*}
			\mathcal{F}(\boldw)=\int_Y M(y,\boldw(y))\dy.
		\end{equation*}
		that
		\begin{equation}\label{fStCom}
		\begin{split}
			f^*(\boldxi)&=\sup_{\boldeta\in\eR^{d\times N}}\left\{\boldxi\cdot\boldeta-\inf_{\boldW\in G(Y)}\mathcal{F}(\boldeta+\boldW)\right\}\\
&=\sup_{\boldeta\in\eR^{d\times N}}\left\{-\inf_{\boldW\in G(Y)}\left\{ \mathcal{F}(\boldeta+\boldW)-\int_{Y}\boldxi\cdot(\boldeta+\boldW(y))\dy\right\}\right\}\\&=-\inf_{\boldeta\in\eR^{d\times N}}\left\{\inf_{\boldW\in G(Y)}\left\{\mathcal{F}(\boldeta+\boldW)-\int_{Y}\boldxi\cdot(\boldeta+\boldW(y))\dy\right\}\right\}\\
&=-\inf_{\boldV\in \eR^{d\times N}\oplus G(Y)}\left\{\mathcal{F}(\boldV)-\int_{Y}\boldxi\cdot\boldV(y)\dy\right\}.
			\end{split}
		\end{equation}
		Expression~\eqref{FSt} is a consequence of Lemma~\ref{Lem:Duality} applied on a functional $\mathcal{F}$. First, we observe that $\mathcal{F}$ is closed or equivalently, whenever $\boldW^k\rightarrow\boldW$ in $L^M(Y;\eR^{d\times N})$ then
		\begin{equation}\label{FClo}
			\liminf_{k\rightarrow\infty}\mathcal{F}(\boldW^k)\geq\mathcal{F}(\boldW).
		\end{equation}
		Obviously $\boldW^k\rightarrow\boldW$ in $L^M_{per}(Y;\eR^{d\times N})$ implies $\boldW^k\rightarrow\boldW$ in $L^1_{per}(Y;\eR^{d\times N})$. In order to show~\eqref{FClo} it suffices to apply the lower semicontinuity of integral functionals with a Carath\'eodory integrand, see \cite[Theorem 4.2]{G03}. Moreover, $\mathcal{F}$ is continuous at $\bzero\in G$, which is a consequence of~\eqref{SSS}. The conjugate functional $\mathcal{F}^*$ to $\mathcal{F}$ is given by
		\begin{equation*}
			\mathcal{F}^*(\boldW^*)=\int_Y M^*(y,\boldW^*(y))\dy
		\end{equation*}
		according to~\eqref{SST}.
		Therefore by Lemma~\ref{Lem:Duality} we get from~\eqref{fStCom}
		\begin{equation*}
			f^*(\boldxi)=\inf_{\boldW^*\in (\eR^{d\times N}\oplus G(Y))^\bot}\int_{Y}M^*(y,\boldW^*(y)+\boldxi)\dy\text{ for all }\boldxi\in\eR^{d\times N}.
		\end{equation*}
		Finally, to conclude~\eqref{FSt} we need to show that
		\begin{equation*}
			\left(\eR^{d\times N}\oplus G(Y)\right)^\bot=\left\{\boldW^*\in G^\bot(Y): \int_Y \boldW^*(y)\dy=0\right\}=:(G^\bot(Y))_0.
		\end{equation*}
Obviously $(G^\bot(Y))_0\subset \left(\eR^{d\times N}\oplus G(Y)\right)^\bot$. In order to get the opposite inclusion, we choose $\boldW^*\in(\eR^{d\times N}\oplus G(Y))^\bot$. Hence by the definition of the annihilator $\int_Y \boldW^*\cdot(\boldeta+\boldW)\dy=0$ for any $\boldeta\in\eR^{d\times N}$ and $\boldW\in G(Y)$. We infer $\int_Y \boldW^*=0$ by setting $\boldW=0$, $\boldeta=\int_Y\boldW^*$ whereas $\boldW^*\in G^\bot(Y)$ follows by setting $\boldeta=0$.
	\end{proof}
\end{Lemma}

The ${\mathcal N}$--functions $f$ and $f^*$ indicate the growth and coercivity properties of the operator $\hat\boldA$ as it is stated among other properties of $\hat\boldA$ in the following lemma.
\begin{Lemma}\label{Lem:2.10}Let the operator $\boldA$ satisfy \ref{AO}--\ref{AF} and the ${\mathcal N}$--function $M$ satisfy \ref{MO}--\ref{M4}. Then we have:
\begin{enumerate}[label=(\^{A}\arabic*)]
	\item \label{AhO} There is a constant $c>0$ such that for all $\boldxi\in\eR^{d\times N}$
		\begin{equation*}
			\hat\boldA(\boldxi)\cdot\boldxi\geq c(f(\boldxi)+f^*(\hat\boldA(\boldxi))).
		\end{equation*}
	\item \label{AhT} For all $\boldxi,\boldeta\in\eR^{d\times N}$, $\boldxi\neq\boldeta$
		\begin{equation*}
			(\hat\boldA(\boldxi)-\hat\boldA(\boldeta))\cdot(\boldxi-\boldeta)> 0.
		\end{equation*}
	\item\label{AhTh} $\hat\boldA$ is continuous on $\eR^{d\times N}$.
\end{enumerate}
\begin{proof}
Let $\boldw$ be a weak solution of cell problem~\eqref{CellPr} corresponding to $\boldxi\in\eR^{d\times N}$, which exists due to Lemma~\ref{Lem:cellex}. Then it follows that
\begin{equation}\label{BelEst}
	\begin{split}
	\hat{\boldA}(\boldxi)\cdot\boldxi&=\int_Y\boldA(y,\boldxi+\nabla \boldw(y))\dy\cdot\boldxi=\int_Y\boldA(y,\boldxi+\nabla \boldw(y))\cdot(\boldxi+\nabla \boldw(y))\dy\\
	&\geq c\int_Y M(y,\boldxi+\nabla \boldw(y))+M^*(y,\boldA(y,\boldxi+\nabla \boldw(y))\dy.
	\end{split}
\end{equation}
Since $\boldw$ is the weak solution to~\eqref{CellPr}, we get from~\eqref{CellPrWF} in a standard way using~\ref{ATh} and the Young inequality that $\boldA(\cdot,\boldxi+\nabla\boldw)\in L^{M^*}_{per}(Y;\eR^{d\times N})$. Moreover, as identity~\eqref{CellPrWF} is satisfied for all $\bphi\in V^M_{per}(Y;\eR^N)$, it is obviously fulfilled for all $\bphi\in W^1_{per}E^M(Y;\eR^N)$. Therefore we have $\boldA(\cdot,\boldxi+\nabla\boldw)\in G^\bot(Y)$. Consequently, regarding~\eqref{AhDef} we obtain by Lemma~\ref{Lem:FSt} that
\begin{equation}\label{MStBoundMod}
	\int_Y M^*(y,\boldA(y,\boldxi+\nabla\boldw(y)))\dy\geq f^*(\hat\boldA(\boldxi)).
\end{equation}
This combined with~\eqref{BelEst} leads to the first part of the estimate in~\ref{AhO}. It remains to  justify that
\begin{equation}\label{justi}
\int_Y M(y,\boldxi+\nabla \boldw(y))\dy \ge \inf_{\bphi \in W^{1}_{per}E^M(Y; \eR)} \int_Y M(y,\boldxi+\nabla \bphi(y))\dy,
\end{equation}
as the rest then follows from the definition of $f$ and~\eqref{BelEst}. However, here we have to face the density problem, which we overcome by using the constructive approach when dealing with the solution. Thus the remaining part of this paragraph will be devoted to the proof of~\eqref{justi}.

We use the fact that $\boldw$ is in fact a modular limit of properly chosen sequence. Indeed, it follows from the construction of the solution  in Theorem~\ref{Thm:WSExistMHoeldC} that there exists a sequence $\{\boldw^k\}_{k=1}^{\infty} \subset W^1_{per}E^M(Y;\eR^N)$ such that
\begin{alignat}{2}
\nabla \boldw^k &\WSCon \nabla \boldw &&\textrm{ in } L^M(Y;\eR^{d\times N}),\label{MB1}\\
\nabla \boldw^k &\SCon \nabla \boldw &&\textrm{ a.e. in } Y,\label{MB9}\\
\boldA(\cdot,\boldxi+\nabla\boldw^k) &\WSCon \boldA(\cdot,\xi+\nabla \boldw) &&\textrm{ in } L^{M^*}(Y;\eR^{d\times N}),\label{MB3}\\
\lim_{k\to \infty} \int_Y \boldA(y,\boldxi +\nabla \boldw^k)\cdot \nabla \boldw^k\dy&\le \int_Y \boldA(y,\boldxi +\nabla \boldw)\cdot \nabla \boldw\dy.&&\label{MB2}
\end{alignat}
Therefore, denoting $\boldW_{\lambda}:=\nabla\boldw \chi_{\{|\nabla \boldw|\le \lambda\}}$, we obtain that (thanks to monotonicity of $\boldA$, the fact that $\boldW_{\lambda}$ is bounded and \eqref{MB1}--\eqref{MB2})
$$
\begin{aligned}
\lim_{\lambda \to \infty}\lim_{k\to \infty} &\int_{Y} \left| (\boldA(y,\boldxi+\nabla\boldw^k)-\boldA(y,\boldxi+\boldW_\lambda))\cdot (\nabla \boldw^k-\boldW_\lambda)\right|\dy\\
&=\lim_{\lambda \to \infty}\lim_{k\to \infty} \int_{Y}  (\boldA(y,\boldxi+\nabla\boldw^k)-\boldA(y,\boldxi+\boldW_\lambda))\cdot (\nabla \boldw^k-\boldW_\lambda)\dy\\
&\leq\lim_{\lambda \to \infty}\int_{Y}  (\boldA(y,\boldxi+\nabla\boldw)-\boldA(y,\boldxi+\boldW_\lambda))\cdot (\nabla \boldw-\boldW_\lambda)\dy\\
&=\lim_{\lambda \to \infty}\int_{Y}  \boldA(y,\boldxi+\nabla\boldw)\cdot \nabla \boldw \chi_{\{|\nabla \boldw|>\lambda\}}\dy=0,
\end{aligned}
$$
where the last equality follows from the fact that $\boldA(\cdot,\boldxi+\nabla\boldw)\cdot \nabla \boldw\in L^1(Y)$. Hence, evidently for any $\varphi\in L^{\infty}(Y)$ we deduce that
$$
\begin{aligned}
\left| \lim_{\lambda \to \infty}\lim_{k\to \infty} \int_{Y} (\boldA(y,\boldxi+\nabla\boldw^k)-\boldA(y,\boldxi+\boldW_\lambda))\cdot (\nabla \boldw^k-\boldW_\lambda)\varphi\dy\right|=0.
\end{aligned}
$$
Hence, it follows from \eqref{MB1}--\eqref{MB3} that
$$
\begin{aligned}
0&=\left| \lim_{k\to \infty} \int_{Y} \boldA(y,\boldxi+\nabla\boldw^k)\cdot \nabla \boldw^k \varphi \dy -\lim_{\lambda \to \infty}\int_Y\boldA(y,\boldxi+\nabla\boldw)\cdot \boldW_{\lambda}\varphi+\boldA(y,\boldxi+\boldW_\lambda))\cdot (\nabla \boldw-\boldW_\lambda)\varphi\dy\right|\\
&=\left| \lim_{k\to \infty} \int_{Y} \boldA(y,\boldxi+\nabla\boldw^k)\cdot \nabla \boldw^k \varphi \dy -\int_Y\boldA(y,\boldxi+\nabla\boldw)\cdot \nabla \boldw\varphi\dy\right|.
\end{aligned}
$$
Thus, we see that
\begin{equation}\label{UU1}
\boldA(y,\boldxi+\nabla\boldw^k)\cdot (\boldxi+\nabla \boldw^k)\rightharpoonup \boldA(y,\boldxi+\nabla\boldw)\cdot (\boldxi+\nabla \boldw) \textrm{ weakly in }L^1(Y).
\end{equation}
Due to the equivalent characterization of the weak convergence in $L^1$, we see that the sequence $\{\boldA(y,\boldxi+\nabla\boldw^k)\cdot (\boldxi+\nabla \boldw^k)\}_{k=1}^\infty$ is uniformly equi-integrable. Using also~\ref{ATh}, we see that also $\{M(y,\boldxi +\nabla \boldw^k)\}_{k=1}^\infty$ is uniformly equi-integrable. Therefore, it follows from the Vitali theorem and~\eqref{MB9} that
$$
\lim_{k\to \infty}\int_Y M(y,\boldxi+\nabla \boldw^k)\dy = \int_Y M(y,\boldxi+\nabla \boldw)\dy.
$$
Consequently, since $\boldw^k \in W^1_{per}E^M(Y;\eR^N)$ we see that~\eqref{justi} holds, which finishes the proof of~\ref{AhO}.

In order to show~\ref{AhT} we fix $\boldxi_1,\boldxi_2\in\eR^{d\times N}, \boldxi_1\neq\boldxi_2$ and find corresponding weak solutions of the cell problem $\boldw_1$ and $\boldw_2$. One obtains (see also appendix)
\begin{equation*}
\int_Y \boldA(y,\boldxi_i+\nabla \boldw_i(y))\cdot\nabla \boldw_j(y)\dy=0\text{ for } i,j=1,2
\end{equation*}
in the same way as~\eqref{TestCPrSol} was shown. Then it follows that
\begin{equation*}
	\begin{split}
	(\hat\boldA(\boldxi_1)-\hat\boldA(\boldxi_2))\cdot(\boldxi_1-\boldxi_2)&=\int_Y(\boldA(y,\boldxi_1+\nabla \boldw_1)-\boldA(y,\boldxi_2+\nabla \boldw_2))\cdot(\boldxi_1-\boldxi_2)\dy\\
	&=\int_Y(\boldA(y,\boldxi_1+\nabla \boldw_1)-\boldA(y,\boldxi_2+\nabla \boldw_2))\cdot(\boldxi_1+\nabla \boldw_1-\boldxi_2-\nabla \boldw_2)\dy>0
	\end{split}
\end{equation*}
by~\ref{ATh}.\\
To show~\ref{AhTh} we consider $\{\boldxi^k\}_{k=1}^\infty$ such that $\boldxi^k\rightarrow\boldxi$ in $\eR^{d\times N}$ as $k\rightarrow\infty$, a corresponding sequence of weak solutions of the cell problems $\{\boldw^k\}_{k=1}^\infty$ and $\boldw$ corresponding to $\boldxi$. Then we have for an arbitrary but fixed $\boldeta\in \eR^{d\times N}$ that
\begin{equation*}
	(\hat\boldA(\boldxi^k)-\hat\boldA(\boldxi))\cdot\boldeta=\int_Y (\boldA(y,\boldxi^k+\nabla \boldw^k)-\boldA(y,\boldxi+\nabla \boldw))\cdot\boldeta\dy\rightarrow 0
\end{equation*}
as $k\rightarrow\infty$ by~\eqref{WSCont}. Since $\eR^{d\times N}$ is finite dimensional, we conclude~\ref{AhTh} from the latter convergence.
\end{proof}
\end{Lemma}

\section{Proof of Theorem~\ref{Thm:Main}}\label{SS3}
\subsection{Setting of the problem}
We start this section by formulating and proving some lemmas that will be used in the proof of Theorem~\ref{Thm:Main} which appears in subsection~\ref{SubSec:MainThmProof}.
Let us outline next steps. First, we derive estimates of a weak solution $\ueps$~of~\eqref{StudPr} and corresponding $\Aeps(x):=\boldA\left(\frac{x}{\varepsilon},\nabla\ueps\right)$ that are uniform with respect to $\varepsilon\in (0,1)$. Then we extract a sequence $\{\boldu^{\varepsilon_k}\}_{k=1}^\infty$ such that $\{\nabla \boldu^{\varepsilon_k}\}_{k=1}^\infty$ converges weakly$^*$ to some $\nabla \boldu$ in $L^{m_1}(\Omega;\eR^{d\times N})$ and a weakly$^*$ convergent sequence $\{\boldA^{\varepsilon_k}\}_{k=1}^\infty$ with a limit $\bar\boldA\in L^{m_2^*}(\Omega;\eR^{d\times N})$. Then we show that the sequence $\{\nabla \boldu^{\varepsilon_k}\}_{k=1}^\infty$ converges weakly$^*$ two--scale to $\nabla \boldu+\boldU$ in $L^{m_1}(\Omega\times Y;\eR^{d\times N})$ and $\{\boldA^{\varepsilon_k}\}_{k=1}^\infty$ converges weakly$^*$ two--scale to $\boldA^0$ in $L^{m_2^*}(\Omega\times Y;\eR^{d\times N})$. Consequently, we apply the weak$^*$ two--scale semicontinuity of convex functionals to improve the regularity of limit functions, i.e., we obtain $\nabla \boldu\in L^f(\Omega;\eR^{d\times N})$ and $\bar\boldA=\int_Y\boldA^0\in L^{f^*}(\Omega;\eR^{d\times N})$. This ensures that $\int_\Omega\bar\boldA\cdot\nabla \boldu\dx$ is meaningful. Then we employ a variant of the Minty trick for nonreflexive function spaces to identify the limit $\bar\boldA$.

First, we formulate the lemma concerning the existence and uniqueness of a solution to problem~\eqref{StudPr} for an arbitrary but fixed $\varepsilon$. The detailed proof in case~\ref{MTw}  or \ref{EMD} is stated in the appendix, see Theorem~\ref{Thm:WSExistMHoeldC}. For the existence proof under condition~\ref{MTw2} we refer to \cite{GMW12}. We denote $M^\varepsilon(x,\boldxi)=M\left(\frac{x}{\varepsilon},\boldxi\right)$.
\begin{Lemma}\label{Lem:WSExist}
 Let $\Omega\subset\Rd$ be a bounded domain, the operator $\boldA$ satisfy \ref{AO}--\ref{ATh} and the $\mathcal{N}$--function~$M$ satisfy \ref{MO}--\ref{M4} and one of ~\ref{MTw}--\ref{EMD} hold. Then for fixed $\varepsilon\in(0,1)$ there exists a unique weak solution of problem~\eqref{StudPr}, which is a function $\ueps\in V_0^{M^\varepsilon}$ such that
\begin{equation}\label{WFEpsPr}
	\int_{\Omega} \boldA\left(\frac{x}{\varepsilon},\nabla\ueps(x)\right)\cdot\nabla\bphi(x)\dx=\int_\Omega \boldF(x)\cdot\nabla\bphi(x)\dx \qquad \textrm{ for all } \bphi\in V_0^{M^\varepsilon}.
\end{equation}
\end{Lemma}

\begin{Lemma}\label{Lem:AprBound}
	Let the assumptions of Lemma~\ref{Lem:WSExist} be satisfied and $\ueps$ be a weak solution of problem~\eqref{StudPr}.
	Then $\{\Aeps\}_{0<\varepsilon< 1}$ is bounded in $L^{m_2^*}(\Omega;\eR^{d\times N})$ and $\{\ueps\}_{0<\varepsilon< 1}$ is bounded in $V_0^{m_1}$ and we have the estimate
\begin{equation}\label{AEIneq}
			\int_{\Omega}\frac{1}{2} m_1(|\nabla\ueps|)+m_2^*(|\Aeps|)\dx\leq c\int_{\Omega} M^\varepsilon\left(x,\nabla\ueps\right)+(M^\varepsilon)^*\left(x,\Aeps\right)\dx\leq C(\|\boldF\|_{\infty}, m_1^*).
		\end{equation}

	\begin{proof} We set $\bphi:=\ueps$ in~\eqref{WFEpsPr} to obtain
		\begin{equation}\label{WeakFormTestSol}
			\int_{\Omega}\Aeps\cdot\nabla\ueps \dx=\int_{\Omega}\boldF\cdot\nabla\ueps\dx.
		\end{equation}
		Using~\eqref{WeakFormTestSol}, \ref{ATh}, the Young inequality, the convexity of $M$ and the fact that the constant $c\leq 1$, which is an obvious consequence of the Young inequality, it follows that
		\begin{equation*}
			c\int_{\Omega} M^\varepsilon(x,\nabla\ueps)+(M^\varepsilon)^*(x,\Aeps)\dx\leq \int_{\Omega}(M^\varepsilon)^*\left(x,\frac{2}{c}\boldF\right)+\frac{c}{2}M^\varepsilon(x,\nabla\ueps)\dx.
		\end{equation*}
		Consequently, employing~\ref{MTwo} we obtain
		\begin{equation*}
			c\int_{\Omega}\frac{1}{2} m_1(|\nabla\ueps|)+m_2^*(|\Aeps|)\dx\leq c\int_{\Omega}\frac{1}{2} M^\varepsilon\left(x,\nabla\ueps\right)+(M^\varepsilon)^*\left(x,\Aeps\right)\dx\leq \int_{\Omega}m_1^*\left(\frac{2}{c}|\boldF|\right)\dx.
		\end{equation*}
		Due to~\eqref{Assumption:F} the integral on the right hand side is finite and the desired conclusion~\eqref{AEIneq} follows.
	\end{proof}
\end{Lemma}

\begin{Lemma}\label{spanek}
Let the assumptions of Lemma~\ref{Lem:WSExist} be satisfied. In addition and $\ueps$ be a weak solution of problem~\eqref{StudPr} and $\{\varepsilon_j\}_{j=1}^\infty$ be an arbitrary sequence such that $\varepsilon_j\rightarrow 0$ as $j\rightarrow\infty$.
	Then there is a subsequence $\{\varepsilon_{j_k}\}_{k=1}^\infty$, functions $\boldu\in V_0^{m_1}$, $\boldU\in L^{m_1}(\Omega\times Y;\eR^{d\times N})$, $\bar\boldA\in L^{m_2^*}(\Omega;\eR^{d\times N})$ and $\boldA^0\in L^{m_2^*}(\Omega\times Y;\eR^{d\times N})$ such that as $k\rightarrow\infty$ we have the following weak convergence results (the sequences are denoted by $k$ and not by $\varepsilon_{j_k}$ for simplicity)
	\begin{equation}\label{WSConv}
		\begin{alignedat}{2}
			\boldu^k&\WSCon \boldu&&\text{ in }L^{m_1}(\Omega;\eR^N),\\
			\nabla\boldu^k&\WSCon\nabla\boldu&&\text{ in }L^{m_1}(\Omega;\eR^{d\times N}),\\
			\boldA^{k}&\WSCon\bar\boldA&&\text{ in }L^{m_2^*}(\Omega;\eR^{d\times N})
		\end{alignedat}
	\end{equation}
	and the weak$^*$ two--scale convergence results
	\begin{equation}\label{WTSC}
		\begin{alignedat}{2}
		 \nabla\uk&\WTSSCon\nabla \boldu+\boldU &&\text{ in }L^{m_1}(\Omega\times Y;\eR^{d\times N}),\\
		 \boldAk&\WTSSCon\boldA^0 &&\text{ in }L^{m^*_2}(\Omega\times Y;\eR^{d\times N}).
		\end{alignedat}
	\end{equation}
	Moreover,
	for a.a. $x\in\Omega$
	\begin{align}
		\boldU(x,\cdot)&\in \{\nabla \boldw:\boldw\in V_{per}^M\},\label{BUReg}\\
		\boldA^0&(x,\cdot)\in G(Y)^\bot\label{AzDual},\\
		\int_Y \boldA^0(x&,y)\cdot\boldU(x,y)\dy=0.\label{AzOrtPr}
	\end{align}
	Furthermore,
	\begin{align}		
		\boldu\in V_0^f(&\Omega;\eR^N),\label{LimRegU}\\
		\bar\boldA=\int_Y&\boldA^0\dy,\label{AZAv}\\
		\bar\boldA\in L^{f^*}(&\Omega;\eR^{d\times N})\label{bAReg}
	\end{align}
	where $f$ is given by~\eqref{DEF:f} and $f^*$ by~\eqref{FSt}.
	The function $\bar\boldA$ satisfies
	\begin{equation}\label{WFLim}
		\int_\Omega\bar\boldA\cdot\nabla\bphi=\int_\Omega \boldF\cdot\nabla\bphi
	\end{equation}
	for all $\bphi\in C^\infty_c\left(\Omega;\eR^N\right)$.
	\begin{proof}
	The convergences~in~\eqref{WSConv} are a direct consequence of the uniform estimates from Lemma~\ref{Lem:AprBound} and the Poincar\'e type inequality, c.f. \cite[Section 2.4]{G79}.
		The convergence~\eqref{WTSC}$_1$ is a consequence of~\eqref{WSConv}$_1$ and Lemma~\ref{Lem:Facts2S}~\ref{F2SSi}, which also yields for almost all $x\in \Omega$
		\begin{equation}\label{OrtProp}
			\int_Y\boldU(x,y)\cdot\bpsi(y)\dy=0 \qquad \textrm{ for all } \bpsi\in C^\infty_{per}(Y;\eR^{d\times N}), \div \bpsi=0,
		\end{equation}
		whereas~\eqref{WTSC}$_2$ follows by Lemma~\ref{Lem:Facts2S}~\ref{F2SFi} due to Lemma~\ref{Lem:AprBound}. Moreover, \eqref{AzOrtPr} follows from Lemma~\ref{Lem:Facts2S}~\ref{F2SS}, \eqref{WSConv}$_2$ and \eqref{WTSC}$_2$.

The convergence result~\eqref{WTSC}$_1$ and the uniqueness of weak$^*$ limit, the weak lower semicontinuity stated in Lemma~\ref{Lem:Facts2S}~\ref{F2SSi} and the uniform estimate~\eqref{AEIneq} imply
	\begin{equation}\label{LimIneq}
		\int_\Omega\int_Y M(y,\nabla \boldu +\boldU)+M^*(y,\boldA^0)\dy\dx\le\liminf_{k\rightarrow\infty}\int_{\Omega}  M^{\varepsilon_k}(x,\nabla \uk (x))+(M^*)^{\varepsilon_k}(x,\boldA^{k}(x)) \dx <\infty.
	\end{equation}

We obtain from~\eqref{LimIneq} the existence of a measurable set $\bar S\subset\Omega$ such that $|\Omega\setminus \bar S|=0$ and for all $x\in \bar S$ $\int_Y M(y,\nabla \boldu(x)+\boldU(x,y))\dy<\infty$, which implies $\boldU(x,\cdot)\in L^M(Y;\eR^{d\times N})$. In addition, it follows from~\eqref{OrtProp} that there exists $\boldw(x,\cdot) \in W^{1,1}_{per}(Y;\eR^N)$ such that $\nabla_y \boldw(x,y)=\boldU(x,y)$. Therefore the estimate~\eqref{LimIneq} gives $\nabla_y \boldw(x,\cdot)\in L^M(Y;\eR^{d\times N})$. Accordingly, we have that $\boldw(x,\cdot)\in V_{per}^M$. Thus by Lemma~\ref{Lem:infjemin} and the definition of function $f$, see~\eqref{DEF:f}, we conclude
\begin{equation*}
	\int_Y M(y,\nabla \boldu(x)+\nabla_y \boldw(x,y)) \dy\geq \inf_{\boldv\in W^1_{per}E^M(Y;\eR^N)} \int_{Y} M(y,\nabla \boldu(x) + \boldv(y))\dy = f(\nabla \boldu(x)).
\end{equation*}
Hence, integrating the result with respect to $x$ over $\Omega$ and using the estimate~\eqref{LimIneq}, we obtain~\eqref{LimRegU}.

In order to show~\eqref{AzDual} we choose $z\in C^\infty_c(\Omega)$ and $\bpsi\in C^\infty_{per}\left(Y;\eR^N\right)$ and set $\bphi(x):=\varepsilon z(x)\bpsi\left(\frac{x}{\varepsilon}\right)$ in~\eqref{WFEpsPr}. Utilizing~\eqref{WTSC}$_2$ and $Y-$periodicity of $\bpsi$ we arrive at
\begin{equation*}
\begin{split}
		&\int_\Omega\int_Y \boldA^0(x,y)\cdot z(x)\nabla\bpsi(y)\dy\dx = \lim_{k\to \infty}\int_\Omega \boldA^k(x)\cdot z(x)\nabla_y\bpsi\left(\frac{x}{\varepsilon_{j_k}}\right)\dx\\
&=\lim_{k\to \infty}\int_\Omega \boldA^k(x)\cdot \nabla\left( \varepsilon_{j_k}z(x)\bpsi\left(\frac{x}{\varepsilon_{j_k}}\right)\right)\dx-\lim_{k\to \infty}\int_\Omega \boldA^k(x)\cdot \left(\varepsilon_{j_k}\nabla z(x) \otimes \bpsi\left(\frac{x}{\varepsilon_{j_k}}\right)\right)\dx\\
&=\int_\Omega \boldF(x) \cdot z(x)\int_Y \nabla_y\bpsi(y)\dy\dx=0,
\end{split}	
\end{equation*}
which implies that there is a measurable set $\tilde S\subset\Omega$, $|\Omega\setminus\tilde S|=0$ such that for all $x\in\tilde S$
	\begin{equation}\label{BarBAId}
		\int_Y\boldA^0(x,y)\cdot\nabla_y\bpsi(y)\dy=0.
	\end{equation}
Using Theorem~\ref{Thm:ModDensity} we can find for any $\bpsi\in W^1_{per}E^M\left(Y;\eR^N\right)$ a sequence $\{\bpsi^k\}_{k=1}^\infty\subset C^\infty_{per}\left(Y;\eR^N\right)$ such that $\nabla\bpsi^k\ModConvM \nabla\bpsi$. Next, we observe that $\boldA^0(x,\cdot)\in L^{M^*}_{per}\left(Y;\eR^N\right)$ for almost all $x\in \Omega$ due to~\eqref{LimIneq}. Then we set $\bpsi=\bpsi^k$ in~\eqref{BarBAId} and employing Lemma~\ref{Lem:ProdConv} we perform the limit passage $k\rightarrow\infty$ to get~\eqref{BarBAId} for any $\bpsi\in W^1_{per}E^M\left(Y;\eR^N\right)$, which implies~\eqref{AzDual}. In a very similar manner,  we use the approximation of $\boldU(x,\cdot)=\nabla_y \boldw(x,\cdot)$ in the modular topology of $L^M_{per}(Y;\eR^{d\times N})$ to conclude~\eqref{AzOrtPr} from~\eqref{BarBAId}.

Using the expression~\eqref{FSt} for $f^*$, the estimate~\eqref{LimIneq}, \eqref{AzDual} and~\eqref{AZAv}, we get
\begin{equation*}
	\int_\Omega f^*(\bar\boldA(x))\dx\leq \int_\Omega\int_Y M^*(y,\boldA^0(x,y))\dy\dx
<\infty,
\end{equation*}
which is~\eqref{bAReg}.

The identity~\eqref{WFLim} is obtained by performing the limit passage $k\rightarrow \infty$ in~\eqref{WFEpsPr} with $\varepsilon=\varepsilon_{j_k}$ using convergence~\eqref{WSConv}$_2$.
	\end{proof}
\end{Lemma}

The rest of the paper is devoted to the identification of $\bar\boldA$ in~\eqref{WFLim}. Before doing so we state the last auxiliary result.
\begin{Lemma}\label{Lem:OperReg}
	Let the assumption \ref{ATh} hold. Then
	\begin{enumerate}
		\item for any $\boldV\in L^\infty(\Omega\times Y;\eR^{d\times N})$ we have $\boldA(\cdot,\boldV)\in L^\infty(\Omega\times Y;\eR^{d\times N})$,
		\item for any $\boldV\in E^{M_y}(\Omega\times Y;\eR^{d\times N})$ we have $\boldA(\cdot,\boldV)\in E^{M_y^*}(\Omega\times Y;\eR^{d\times N})$ provided~\ref{MTwo} holds.
		\end{enumerate}
	\begin{proof}
		Let us observe that~\ref{ATh} implies
		\begin{equation*}
			|\boldV|\geq c\frac{M^*(\cdot,\boldA(\cdot,\boldV))}{|\boldA(\cdot,\boldV)|}.
		\end{equation*}
		Assume that $\boldV\in L^\infty\left(\Omega\times Y;\eR^{d\times N}\right)$ and $\|\boldA(\cdot,\boldV)\|_{L^\infty}=\infty$, i.e., for any $K>0$ there is a set $S_K\subset\Omega\times Y$, $|S_K|>0$ such that $|\boldA(\cdot,\boldV(\cdot,\cdot))|> K$ on $S_K$. Since $M^*$ is an $N-$function, for any $L>0$ there is $K_L>0$ such that we have $|\boldV|\geq c\frac{M^*(y,\boldA(y,\boldV)}{|\boldA(y,\boldV)|}>L$ on $S_{K_L}$ with $|S_{K_L}|>0$, which contradicts $\boldv\in L^\infty(\Omega\times Y;\eR^{d\times N})$.\\
		By~\ref{ATh} and the Young inequality we obtain for any $t\geq 0$ and $\boldV\in E^{M_y}\left(\Omega\times Y;\eR^{d\times N}\right)$ that
		\begin{equation*}
		\begin{split}
		c&\int_\Omega\int_Y M(y,\boldV)+M^*(y,t\boldA(y,\boldV))\dy\dx\leq \int_\Omega\int_Y t\boldA(y,\boldV))\cdot\boldV\dy\dx\\&\leq c\int_\Omega\int_Y M\left(y,\frac{2}{c}\boldV\right)+\frac{c}{2}M^*(y,t\boldA(y,\boldV))\dy\dx.
		\end{split}
		\end{equation*}
		Hence we infer $\int_\Omega\int_Y M^*(y,t\boldA(y,\boldV))\dy\dx\leq {2}\int_\Omega\int_Y M\left(y,\frac{2}{c}\boldV\right)\dy\dx$ and the latter integral is finite by Lemma~\ref{Lem:EMChar}. We note that~\eqref{MLocBound} holds since we assume~\ref{MTwo}. We also utilize Lemma~\ref{Lem:EMChar} to conclude that $\boldA(y,\boldV)\in E^{M^*_y}\left(\Omega\times Y;\eR^{d\times N}\right)$.
	\end{proof}
\end{Lemma}

\def\bpsi{\boldsymbol{\psi}}
\def\boldAEpsk{\boldA_{\varepsilon_k}}
\subsection{Identification of the homogenized problem}\label{SubSec:MainThmProof}
In this final part we identify $\bar\boldA$. Through this section we always assume that all assumptions of Lemma~\ref{spanek} are satisfied and we consider the sequence of solutions $\boldu^k$ according to Lemma~\ref{spanek}.\\
\textbf{Step 1}: We show the following identity
\begin{equation}\label{LimIntProd}
	\lim_{k\rightarrow\infty}\int_\Omega \boldA^k\cdot\nabla \uk\dx=\int_\Omega \bar\boldA\cdot\nabla \boldu\dx.
\end{equation}
To show it, we first deduce the validity of the following identity
\begin{equation}\label{HomPrWeakForm}
\int_\Omega \bar{\boldA}\cdot\nabla \boldu\dx=\int_\Omega\boldF\cdot\nabla\boldu\dx.
\end{equation}
If~\ref{MTw} or \ref{EMD} is fulfilled, the according Lemma~\ref{Thm:ModDensity}, we can find a sequence $\{\boldu^n\}_{n=1}^\infty\subset C^\infty_c\left(\Omega;\eR^N\right)$ such that $\boldu^n\ModConv{f} \boldu$ as $n\rightarrow\infty$. Then we set $\bphi=\boldu^n$ in~\eqref{WFLim} and using Lemma~\ref{Lem:ProdConv} we conclude~\eqref{HomPrWeakForm}. Finally, if~\ref{MTw2} holds, we find for each $k\in\eN$ a sequence $\{u^{k,n}\}_{n=1}^\infty\subset C^\infty_c(\Omega)$ such that $\nabla u^{k,n}\ModConv{f} \nabla T_k(u)$ as $n\rightarrow\infty$, where the truncation operator $T_k$ was introduced in the proof of Lemma~\ref{Lem:MConvEquiv}. Then we set $\varphi=u^{k,n}$ in~\eqref{WFLim} and using Lemma~\ref{Lem:ProdConv} we deduce
\begin{equation*}
	\int_\Omega \bar\boldA\cdot\nabla T_k(u)\dx=\int_\Omega \boldF\cdot\nabla T_k(u)\dx.
\end{equation*}
Applying Lemma~\ref{Lem:ModConvGrTrunc} we deduce~\eqref{HomPrWeakForm}. Then it follows from~\eqref{WeakFormTestSol} using~\eqref{WSConv}$_1$ and~\eqref{HomPrWeakForm} that
\begin{equation*}
	\lim_{k\rightarrow\infty}\int_\Omega \boldA^k\cdot\nabla \uk\dx=\lim_{k\rightarrow\infty}\int_\Omega \boldF\cdot\nabla \uk\dx=\int_\Omega \boldF\cdot\nabla \boldu\dx=\int_\Omega \bar\boldA\cdot\nabla \boldu\dx,
\end{equation*}
which concludes~\eqref{LimIntProd}.\\
\textbf{Step 2}: We show that the following inequality holds for all $\boldV\in C^\infty_c(\Omega;C^\infty_{per}(Y;\eR^{d\times N}))$.
 \begin{equation}\label{LimMonIneq}
	0\leq\int_\Omega\int_Y (\boldA^0(x,y)-\boldA(y,\boldV(x,y)))\cdot(\nabla \boldu(x)+\boldU(x,y)-\boldV(x,y))\dy\dx
\end{equation}
 Let us choose $\boldV\in C^\infty_c(\Omega;C^\infty_{per}(Y;\eR^{d\times N}))$. Then according to Lemma~\ref{Lem:OperReg} we obtain $\boldA(\cdot,\boldV)\in L^\infty(\Omega\times Y;\eR^{d\times N})\subset E^{m^*_1}(\Omega\times Y;\eR^{d\times N})\subset E^{m^*_2}(\Omega\times Y;\eR^{d\times N})$. Moreover, $\boldA(\cdot,\boldV)$ is obviously Carath\'eodory. Then for $\boldV^k(x)=\boldV(x,x\varepsilon_k^{-1})$ and $\tilde{\boldA}^k(x):=\boldA(x\varepsilon_k^{-1},\boldV^k(x))$ we obtain
\begin{equation}\label{S2SC}
	\begin{alignedat}{2}
		\boldV^k&\STSCon \boldV &&\text{ in }E^{m_i}(\Omega\times Y;\eR^{d\times N}),\\
		\tilde{\boldA}^k&\STSCon \boldA(\cdot,\boldV(\cdot,\cdot)) &&\text{ in }E^{m_i^*}(\Omega\times Y;\eR^{d\times N}), i=1,2,
	\end{alignedat}
\end{equation}
as $k\rightarrow\infty$ by Lemma~\ref{Lem:Facts2S}~\ref{F2SFir}. From~\ref{AF} we get
\begin{equation*}
	\begin{split}
	0\leq& \int_\Omega (\boldA^k(x)-\tilde\boldA^k(x))\cdot(\nabla\uk(x)-\boldV^k(x))\dx=\int_\Omega \boldA^k(x)\cdot\nabla\uk(x)\dx-\int_\Omega \boldA^k(x)\cdot\boldV^k(x)\dx\\&-\int_\Omega \tilde\boldA^k(x)\cdot\nabla\uk(x)\dx+\int_\Omega \tilde\boldA^k(x)\cdot\boldV^k(x)\dx =I_k-II_k-III_k+IV_k.
	\end{split}
\end{equation*}
Now, want to perform the passage $k\rightarrow\infty$. Using~\eqref{LimIntProd} we obtain that
\begin{equation*}
	\begin{split}
	\lim_{k\rightarrow \infty}I_k&=\int_\Omega \bar\boldA\cdot \nabla \boldu\dx.
	\end{split}
\end{equation*}
Employing properties~\eqref{AZAv} and \eqref{AzOrtPr} yields
\begin{equation*}
	\lim_{k\rightarrow \infty}I_k=\int_\Omega \int_Y \boldA^0 \cdot \nabla \boldu\dy\dx\\
	= \int_\Omega \int_Y \boldA^0 \cdot (\nabla \boldu+\boldU)\dy\dx.
\end{equation*}
It follows from~\eqref{WTSC}$_2$, \eqref{S2SC}$_1$ and Lemma~\ref{Lem:Facts2S}~\ref{F2ST} that
\begin{equation*}
	\lim_{k\rightarrow\infty}II_k=\int_\Omega\int_Y \boldA^0\cdot\boldV\dy\dx,
\end{equation*}
whereas~\eqref{WTSC}$_1$,\eqref{S2SC}$_2$ and Lemma~\ref{Lem:Facts2S}~\ref{F2ST} imply
\begin{equation*}
	\lim_{k\rightarrow\infty}III_k=\int_\Omega\int_Y \boldA(y,\boldV(x,y))\cdot(\nabla \boldu(x)+\boldU(x,y))\dy\dx.
\end{equation*}
Finally, from~\eqref{S2SC} we deduce
\begin{equation*}
	\lim_{k\rightarrow\infty}IV_k=\int_\Omega\int_Y \boldA(y,\boldV(x,y))\cdot\boldV(x,y)\dy\dx.
\end{equation*}
Hence one obtains~\eqref{LimMonIneq}.\\
\textbf{Step 3}:
The goal is to show that $\boldV\in C^\infty_c(\Omega;C^\infty_{per}(Y;\eR^{d\times N}))$ in~\eqref{LimMonIneq} can be substituted by $\boldV\in L^\infty(\Omega\times Y;\eR^{d\times N})$. Let us fix an arbitrary function $\boldV \in L^\infty(\Omega\times Y;\eR^{d\times N})$. We first consider a sequence $\{K^m\}_{m=1}^\infty$ of compact subsets of $\Omega$ such that $K^1\subset K^2\subset\ldots\Omega$ and $\bigcup_{m=1}^\infty K^m=\Omega$. Obviously, defining $\boldV^m:=\boldV\chi_{K^m}$ for every $m\in\eN$ we have that all $\boldV^m$'s are compactly supported in $\Omega$ and
\begin{equation}\label{VMBound}
	\|\boldV^m\|_{L^\infty(\Omega\times Y)}\leq \|\boldV\|_{L^\infty(\Omega\times Y)}\text{ for all }m\in\eN.
\end{equation}
Next, we observe that~\eqref{VMBound} implies the existence of a positive constant $c$ such that
\begin{equation}\label{AVMUnifBound}
	\|\boldA(\cdot, \boldV^m)\|_{L^\infty(\Omega\times Y)}\leq c \ \text{ for all }m\in\eN.
\end{equation}
Assuming on the contrary that $\{\boldA(\cdot,\boldV^m)\}_{m=1}^\infty$ is unbounded, we have for arbitrary $K>0$ the existence of $m_K>0$ and $S_K\subset\Omega\times Y$ with $|S_K|>0$ such that $|\boldA(\cdot,\boldV^{m_K})|>K$ on $S_K$. As $M$ is an $N-$function, for a chosen $C>0$ there is $R>0$ such that $\frac{M^*(y,\boldxi)}{|\boldxi|}>C$ for any $|\boldxi|\geq R$. Thus for the choice $C=\|\boldV\|_{L^\infty(\Omega\times Y)}$ we find $m_{R}$ and $S_{R}\subset\Omega\times Y$  with $|S_{R}|>0$ such that for $(x,y)\in S_{R}$ we obtain using~\ref{ATh}
\begin{equation*}
	C< \frac{M^{*}(y,\boldA(y,\boldV^{m_{R}}))}{|\boldA(y,\boldV^{m_{R}})|}\leq |\boldV^{m_{R}}|\leq \sup_{m\in\eN}\|\boldV^m\|_{L^\infty(\Omega\times Y)}\leq C,
\end{equation*}
which is a contradiction and~\eqref{AVMUnifBound} is shown. Combining~\ref{MTwo} with~\eqref{VMBound} and~\eqref{AVMUnifBound} we get
\begin{equation*}
\begin{split}
	\int_\Omega\int_Y& M(y,\boldV^m)+M^*(y,\boldA(y,\boldV^m)\dy\dx\leq \int_\Omega\int_Y m_2(|\boldV^m|)+m_1^*(|\boldA(y,\boldV^m)|)\dy\dx\\
	&\leq \int_\Omega\int_Y m_2(\|\boldV^m\|_{L^\infty(\Omega\times Y)})+m_1^*(\|\boldA(\cdot,\boldV^m)\|_{L^\infty(\Omega\times Y)})\leq c.
\end{split}
\end{equation*}
Hence $\{\boldV^m\}_{m=1}^\infty$ and $\{\boldA(\cdot,\boldV^m)\}_{m=1}^\infty$ are uniformly integrable by Lemma~\ref{Lem:UnifIntegr}. Furthermore, it follows from the definition of $\boldV^m$ and the properties of $\boldA$ that $\boldV^m\to\boldV$ and $\boldA(\cdot,\boldV^m)\to\boldA(\cdot,\boldV)$ in measure as $m\to\infty$. Consequently, we get by Lemma~\ref{Lem:MConvEquiv} that
\begin{equation}\label{ConvInM}
	\boldV^m\ModConvM\boldV\text{ in }L^{M_y}(\Omega\times Y;\eR^{d\times N}),\ \boldA(\cdot,\boldV^m)\ModConv{M^*}\boldA(\cdot,\boldV) \text{ in }L^{M^*_y}(\Omega\times Y;\eR^{d\times N})\text{ as }m\to\infty.
\end{equation}
Let us consider a standard mollifier $\omega\in C^\infty(\eR^d\times\eR^d)$. Since $\boldV^m$ is supported in $K^m\subset\Omega$ for all $m$, we can find for every $m$ a sequence $\delta^n\to 0$ as $n\to\infty$ such that, defining $\boldV^{m,n}:=\boldV^m*\omega^n$, where $\omega^n(z)=(\delta^n)^{-2d}\omega\left(\frac{z}{\delta^n}\right)$, we have $\boldV^{m,n}\in C^\infty_c(\Omega;C^\infty_{per}(Y))^{d\times N}$. We immediately observe that $\|\boldV^{m,n}\|_{L^\infty(\Omega\times Y)}\leq\|\boldV^m\|_{L^\infty(\Omega\times Y)}$. In the same way as~\eqref{AVMUnifBound} was shown we get that $\|\boldA(\cdot,\boldV^{m,n})\|_{L^\infty(\Omega\times Y)}\leq c(m)$. We also obtain that $\boldV^{m,n}\to\boldV^m$ and $\boldA(\cdot,\boldV^{m,n})\to\boldA(\cdot,\boldV^m)$ in measure as $n\to\infty$ for every $m$. Moreover, for every $m$ the sequences $\{\boldV^{m,n}\}_{n=1}^\infty$ and $\{\boldA(\cdot,\boldV^{m,n})\}_{n=1}^\infty$ are uniformly integrable, which can be shown analogously as above. Consequently, we have for every $m$ that
\begin{equation}\label{ConvInN}
	\boldV^{m,n}\ModConvM\boldV^m\text{ in }L^{M_y}(\Omega\times Y;\eR^{d\times N}),\ \boldA(\cdot,\boldV^{m,n})\ModConv{M^*}\boldA(\cdot,\boldV^m) \text{ in }L^{M^*_y}(\Omega\times Y;\eR^{d\times N})\text{ as }n\to\infty.
\end{equation}
Finally, employing~\eqref{ConvInN}, \eqref{ConvInM} and Lemma~\ref{Lem:ProdConv} we infer from~\eqref{LimMonIneq} that
\begin{equation*}
	0\leq \lim_{m\to\infty}\lim_{n\to\infty}\int_\Omega\int_Y (\boldA^0-\boldA(y,\boldV^{m,n}))\cdot(\nabla \boldu+\boldU-\boldV^{m,n})=\int_\Omega\int_Y (\boldA^0-\boldA(y,\boldV))\cdot(\nabla \boldu+\boldU-\boldV).
\end{equation*}
\textbf{Step 4}: Let us denote for a positive $k$
\begin{equation*}
S_k=\{(x,y)\in\Omega\times Y:|\nabla \boldu(x)+\boldU(x,y)|\leq k\}
\end{equation*}
and $\chi_k$ be the characteristic function of $S_k$. We replace $\boldV\in L^\infty(\Omega\times Y;\eR^{d\times N})$ in~\eqref{LimMonIneq} by $(\nabla \boldu+\boldU)\chi_j+h\boldV\chi_i$ where $0<i<j$ and $h\in(0,1)$ to obtain
\begin{equation*}
	\begin{split}
	0\leq& \int_\Omega\int_Y \boldA^0\cdot(\nabla \boldu+\boldU-(\nabla \boldu+\boldU)\chi_j)\dy\dx\\
	&-\int_\Omega\int_Y\boldA(y,(\nabla \boldu+\boldU)\chi_j+h\boldV\chi_i))\cdot(\nabla \boldu+\boldU-(\nabla \boldu+\boldU)\chi_j)\\
	&-h\int_\Omega\int_Y(\boldA^0-\boldA(y,(\nabla \boldu+\boldU)\chi_j+h\boldV\chi_i))\cdot\boldV\chi_i\dy\dx=I-II+III.
	\end{split}
\end{equation*}

The term $I$ disappears when performing the limit passage $j\rightarrow\infty$ by the Lebesgue dominated convergence theorem and the fact $|\Omega\times Y\setminus S_j|\rightarrow 0$ as $j\rightarrow \infty$. As $(\nabla \boldu+\boldU)\chi_j+h\boldV\chi_i$ is zero in $\Omega\times Y\setminus S_j$, we see that $II=0$ thanks to~\ref{ATh}. After dividing the resulting inequality by $h$ and letting $j\to \infty$ we arrive with the help of Lebesgue dominated convergence theorem at
\begin{equation}\label{IntPos}
	\int_{S_i}(\boldA^0-\boldA(y,\nabla \boldu+\boldU+h\boldV))\cdot\boldV\dy\dx\le 0.
\end{equation}
By~\ref{MTwo} we obtain
\begin{equation}\label{UnifEstOnSI}
\begin{split}
	\int_{S_i} &M^*(y,\boldA(y,\nabla \boldu+\boldU+h\boldV))\dy\dx\leq \int_{S_i} m_1^*(|\boldA(y,\nabla \boldu+\boldU+h\boldV)|)\dy\dx\\
	&\leq |S_i|m_1^*(\|\boldA(\cdot,\nabla \boldu+\boldU+h\boldV)\|_{L^\infty(S_i)})\leq c.
	\end{split}
\end{equation}
The fact that $\|\boldA(\cdot,\nabla \boldu+\boldU+h\boldV)\|_{L^\infty(S_i)}$ is bounded independently of $h\in(0,1)$ is shown in the same way as~\eqref{AVMUnifBound} because
\begin{equation*}
\|\nabla \boldu+\boldU+h\boldV\|_{L^\infty(S_i)}\leq \|\nabla \boldu+\boldU\|_{L^\infty(S_i)}+\|\boldV\|_{L^\infty(\Omega\times Y)}\leq i+\|\boldV\|_{L^\infty(\Omega\times Y)}.
\end{equation*}
Since $\boldA(y,\nabla \boldu+\boldU+h\boldV)\rightarrow\boldA(y,\nabla \boldu+\boldU)$ a.e. in $S_i$ and $\{\boldA(y,\nabla \boldu+\boldU+h\boldV)\}_{h\in(0,1)}$ is uniformly integrable on $S_i$ due to~\eqref{UnifEstOnSI} and Lemma~\ref{Lem:UnifIntegr}, the Vitali theorem implies
\begin{equation*}
	\boldA(y,\nabla \boldu+\boldU+h\boldV) \rightarrow \boldA(y,\nabla \boldu+\boldU)\text{ in }L^1(S_i).
\end{equation*}
Therefore passing to the limit $h\rightarrow 0_+$ in~\eqref{IntPos} we arrive at
\begin{equation*}
	\int_{S_i}(\boldA^0-\boldA(y,\nabla \boldu+\boldU))\cdot\boldV\dy\dx\le 0.
\end{equation*}
Finally, setting
\begin{equation*}
	\boldV=\frac{\boldA^0-\boldA(y,\nabla \boldu+\boldU)}{|\boldA^0-\boldA(y,\nabla \boldu+\boldU)|+1}
\end{equation*}
yields
\begin{equation}\label{PointEq}
	\boldA^0(x,y)=\boldA(y,\nabla \boldu(x)+\boldU(x,y))
\end{equation}
for a.a. $(x,y)\in S_i$. Since $i$ was arbitrary and $|\Omega\times Y\setminus S_i|\rightarrow 0$ as $i\rightarrow \infty$, the equality~\eqref{PointEq} holds a.e. in $\Omega\times Y$.
Moreover, due to the properties~\eqref{BUReg} and~\eqref{AzDual} we obtain that $\boldU(x,\cdot)$ is equal to the gradient of a weak solution of the cell problem~\eqref{CellPr} corresponding to $\boldxi=\nabla \boldu(x)$. Finally, we get by~\eqref{AZAv} and~\eqref{AhDef} that
\begin{equation}\label{BAIdent}
	\bar\boldA(x)=\int_Y\boldA^0(x,y)\dy=\int_Y \boldA(y,\nabla \boldu(x)+\boldU(x,y))\dy=\hat\boldA(\nabla \boldu(x)).
\end{equation}
\textbf{Step 5}:
The existence of a unique weak solution of the problem~\eqref{HPr}, which is a function $\boldu\in V_0^f$ that satisfies
	\begin{equation}\label{HSWF}
			\int_\Omega \hat\boldA(\nabla \boldu)\cdot\nabla\bphi=\int_\Omega\boldF\cdot\nabla\bphi\ \forall\bphi\in V_0^f.
	\end{equation}
We notice that the existence part has been proven in the previous steps. Indeed, in~\eqref{BAIdent} we identified the function $\bar\boldA$, which arises in~\eqref{WFLim}. Then using the density of smooth compactly supported functions in $V^f_0$ we conclude~\eqref{HSWF}.
	In order to show the uniqueness of a weak solution of~\eqref{HPr} we can follow the proof of the uniqueness of a weak solution in Theorem~\ref{Thm:WSExistMHoeldC}.\\
	\textbf{Step 6}:
	Since we know that~\eqref{HPr} possesses a unique solution $\boldu$ and we can extract from any subsequence of $\{\boldu^j\}_{j=1}^\infty$ a subsequence that converges to $\boldu$ weakly in $W^{1,1}_0(\Omega;\eR^N)$), the whole sequence $\{\boldu^j\}_{j=1}^\infty$ converges to $\boldu$ weakly in $W^{1,1}_0(\Omega;\eR^N)$.

\begin{appendix}
\section{Musielak--Orlicz spaces}\label{Ape1}
Assume here that $\Sigma\subset\eR^n$ is a bounded domain and $n\in \mathbf{N}$ is arbitrary. A function $M:\Sigma\times\eR^n\rightarrow[0,\infty)$ is said to be an ${\mathcal N}-$function if it satisfies the following four requirements:
	\begin{enumerate}
		\item $M$ is a Carath\'eodory function such that $M(x,\boldxi)=0$ if and only if $\boldxi=\bzero$. In addition we assume that for almost all $x\in \Sigma$, we have  $M(x,\boldxi)=M(x,-\boldxi)$.
		\item For almost all $x\in \Sigma$ the mapping $\boldxi\mapsto M(x,\boldxi)$ is convex.
		\item For almost all $x\in \Sigma$ there holds $\lim_{\substack{|\boldxi|\rightarrow\infty}}\frac{M(x,\boldxi)}{|\boldxi|}=\infty$.
		\item For almost all $x\in \Sigma$ there holds $\lim_{|\boldxi|\rightarrow 0}\frac{M(x,\boldxi)}{|\boldxi|}=0$.
	\end{enumerate}		
The corresponding complementary ${\mathcal N}$--function $M^*$ to $M$ is defined for $\boldeta\in\eR^n$ and almost all  $x\in \Sigma$ by
\begin{equation*}
		M^*(x,\boldeta):=\sup_{\boldxi\in\eR^n}\{\boldxi\cdot\boldeta-M(x,\boldxi)\}
\end{equation*}
and directly from this  definition, one obtains the generalized Young inequality
\begin{equation}\label{YIneq}
		\boldxi\cdot\boldeta\leq M(x,\boldxi)+M^*(x,\boldeta),
\end{equation}
valid for all $\boldxi,\boldeta\in\eR^n$ and almost everywhere in $\Sigma$. In addition, for $\boldxi:=\nabla_{\boldeta}M^*(x,\boldeta)$, we obtain the equality sign in~\eqref{YIneq}, see \cite[Section 5]{SI69}. Finally, an ${\mathcal N}$-function $M$ is said to satisfy the $\Delta_2$--condition if there exists $c>0$ and a nonnegative function $h\in L^1(\Sigma)$ such that for a.a. $x\in\Sigma$ and all $\boldxi\in\eR^n$
	\begin{equation*}
		M(x,2\boldxi)\leq cM(x,\boldxi)+h(x).
	\end{equation*}

Having introduced the notion of $\mathcal{N}$--function, we can define the generalized Musielak--Orlicz class	$\mathcal{L}^M(\Sigma)$ as a set of all measurable functions $\boldv:\Sigma\rightarrow\eR^n$ in the following way
\begin{equation*}
		\mathcal{L}^M(\Sigma):=\left\{\boldv \in L^1(\Sigma;\eR^n); \; \int_\Sigma M(x,\boldv(x))\dx<\infty\right\}.
\end{equation*}
In general the class $\mathcal{L}^M(\Sigma)$ does not form a linear vector space and therefore, we define the generalized Musielak--Orlicz space $L^M(\Sigma)$ as the smallest linear space containing $\mathcal{L}^M(\Sigma)$. More precisely, we define
\begin{equation*}
		L^M(\Sigma):=\left\{\boldv \in L^1(\Sigma;\eR^n); \; \textrm{ there exists $\lambda>0$ such that }\int_\Sigma M\left(x,\frac{\boldv(x)}{\lambda}\right)\dx<\infty\right\}.
\end{equation*}
It can be shown that $L^M(\Sigma)$ is a Banach space with respect to the Orlicz norm
	\begin{equation*}
		\|\boldv\|_{L^M}:=\sup\left\{\left|\int_{\Sigma}\boldv(x)\boldw(x)\dx\right|\!:\boldw\in L^{M^*}(\Sigma), \int_{\Sigma}M(x,\boldw(x))\dx\leq 1 \right\}
	\end{equation*}
or the equivalent Luxemburg norm
\begin{equation*}
		\|\boldv\|_{L^M}:=\inf\left\{\lambda>0:\int_\Sigma M\left(x,\frac{\boldv(x)}{\lambda}\right)\dx\leq 1\right\}.
\end{equation*}
Moreover, we have the following generalized H\"{o}lder inequality, see \cite[Theorem 4.1.]{SII69},
\begin{equation*}
		\left|\int_\Sigma \boldu\cdot\boldv \dx \right|\leq 2\|\boldu\|_{L^M}\|\boldv\|_{L^{M^*}}
\end{equation*}
valid for all $\boldu\in L^M(\Sigma)$ and all $\boldv\in L^{M^*}(\Sigma)$. It is not difficult to observe directly from the definition (or by Young inequality~\eqref{YIneq}), that
\begin{equation}\label{SSF}
		\|\boldv\|_{L^M}\leq c\left(\int_{\Sigma}M(x,\boldv(x))\dx+1\right),
\end{equation}
with some $c>0$, that can be set $c=1$ if we work with the Orlicz norm. Similarly, for the functional $\mathcal{F}:L^M(\Sigma)\rightarrow\eR$ defined as
\begin{equation*}
		\mathcal{F}(\boldv):=\int_\Sigma M(x,\boldv(x))\dx,
\end{equation*}
we can directly obtain from the definition and due to the convexity of $M$ that if $\|\boldv\|_{L^M}\leq 1$ and the Luxemburg norm is considered then
\begin{equation}
\label{SSS}
\mathcal{F}(\boldv)\leq\|\boldv\|_{L^M}.
\end{equation}
Finally, we also recall the definition of the conjugate functional $\mathcal{F}^*:L^{M^*}(\Sigma)\rightarrow\eR$
$$
\mathcal{F}^*(\boldv^*):= \sup_{\boldv \in L^M(\Sigma)} \left(\int_{\Sigma} \boldv \cdot \boldv^* \dx - \mathcal{F}(\boldv) \right)
$$
and it is not difficult to observe by using the Young inequality that\footnote{Young inequality~\eqref{YIneq} implies $\mathcal{F}^*(\boldv^*)\leq \int_\Sigma M^*(x,\boldv^*(x))\dx$. On the other hand, we have $M^*(\cdot,\boldv^*)=\boldv^*\cdot\boldw-M(\cdot,\boldw)$ for $\boldw(x):=\nabla_{\boldxi}M^*(x,M^*(x,\boldv^*))$, which after integration leads to $\mathcal{F}^*(\boldv^*)\ge \int_\Sigma M^*(x,\boldv^*(x))\dx$ and~\eqref{SST} follows.}
\begin{equation}
\label{SST}
\mathcal{F}^*(\boldv^*) =\int_\Sigma M^*(x,\boldv^*(x))\dx.
\end{equation}

We complete this subsection by recalling the basic functional-analytic facts about the generalized Musielak--Orlicz spaces. For this purpose we define an additional space
\begin{equation*}
E^M(\Sigma):= \overline{\left\{L^{\infty}(\Sigma;\eR^n)\right\}}^{\|\cdot \|_{L^M(\Sigma)}}.
\end{equation*}
The following key lemma summarizes the fundamental properties of the involved function spaces (see e.g. \cite{S05} for details).
\begin{Lemma}[separability, reflexivity]\label{Thm:OrlSpProp}
	Let $M$ be an ${\mathcal N}$--function. Then
	\begin{enumerate}
		\item $E^M(\Sigma)=L^M(\Sigma)$ if and only if $M$ satisfies the $\Delta_2$--condition,
		\item $(E^M(\Sigma))^*=L^{M^*}(\Sigma)$, i.e., $L^{M^*}(\Sigma)$ is a dual space to $E^M(\Sigma)$,
		\item $E^M(\Sigma)$ is separable,
		\item $L^M(\Sigma)$ is separable if and only if $M$ satisfies the $\Delta_2$--condition,
		\item $L^M(\Sigma)$ is reflexive if and only if $M,M^*$ satisfy the $\Delta_2$--condition.
	\end{enumerate}
\end{Lemma}
We see from the above lemma that in some cases we need to face the problem with the density of bounded functions and also the lack of reflexivity and separability properties, that somehow excludes many analytical framework to be used. Thus, in addition to the strong/weak/weak$^*$ topology, we will also work with the modular topology. We say that a sequence $\{\boldv^k\}_{k=1}^\infty\subset L^M(\Sigma)$ converges modularly to $\boldv$ in $L^M(\Sigma)$ if there is $\lambda>0$ such that as $k\rightarrow\infty$
	\begin{equation*}
		\int_{\Sigma} M\left(x,\frac{\boldv^k(x)-\boldv(x)}{\lambda}\right)\dx\rightarrow 0.
	\end{equation*}
We use the notation $\boldv^k\ModConvM\boldv$ for the modular convergence in $L^M(\Sigma)$. The key property of the modular convergence is stated in the following lemma.

\begin{Lemma}\cite[Proposition 2.2.]{GSG08}\label{Lem:ProdConv}
Let $M$ be an ${\mathcal N}$--function and $M^*$ be the conjugate ${\mathcal N}$--function to $M$. Suppose that sequences $\{\boldv^k\}_{k=1}^\infty$ and $\{\boldw^k\}_{k=1}^\infty$ are uniformly bounded in $L^M(\Sigma)$, $L^{M^*}(\Sigma)$ respectively. Moreover, let $\boldv^k\ModConvM\boldv$ and $\boldw^k\ModConv{M^*}\boldw$. Then $\boldv^k\cdot\boldw^k\rightarrow\boldv\cdot\boldw$ in $L^1(\Sigma)$ as $k\rightarrow\infty$.
\end{Lemma}
Finally, we also recall the weak$^*$ lower semicontinuity property of convex functionals. Since in our case,  the ${\mathcal N}$--function $M$ may not  satisfy the $\Delta_2$--condition in general, the spaces  do not have to be  reflexive. However, due to Lemma~\ref{Thm:OrlSpProp}, we see that any $L^M$ always has a separable predual space and consequently any bounded sequence possesses a weakly$^*$ convergent subsequence. This motivates us to introduce the last convergence theorem, that can be obtained by standard  weak lower semicontinuity properties of convex functionals, see e.g. \cite[Theorem 4.5]{G03}, namely:
\begin{Lemma}\label{Lem:SemCon}
	Let $\Omega\subset\Rd$ be open, $Y=(0,1)^d$, $n\in\eN$ and $\Phi:Y\times \eR^n \rightarrow \eR$ satisfy:
\begin{enumerate}[label=(\alph*)]
	\item $\Phi$ is Carath\'eodory,
	\item $\Phi(y,\cdot)$ is convex for almost all $y\in Y$,
	\item $\Phi\geq 0$.
\end{enumerate}
	Then we have the following semicontinuity property: $\boldv^k\WCon \boldv$ in $L^1(\Omega\times Y;\eR^n)$ as $k\rightarrow\infty$ implies
	\begin{equation*}
		\liminf_{k\rightarrow\infty}\int_{\Omega}\int_Y \Phi(y,\boldv^k(x,y))\dy\dx\geq\int_\Omega\int_Y \Phi(y,\boldv(x,y))\dy\dx.
		\end{equation*}
\end{Lemma}

We continue with the characterization of the space $E^M$.
\begin{Lemma}\label{Lem:EMChar}
Let $\Sigma\subset\Rd$ be bounded, $M$ be an ${\mathcal N}$--function such that for all $R>0$
\begin{equation}\label{MLocBound}
	\int_{\Sigma}\sup_{|\boldxi|\leq R}M(x,\boldxi)\dx<\infty.
\end{equation}
Then
\begin{equation*}
	E^M(\Sigma)=\{\boldv\in L^M(\Sigma):\; \textrm{ for all } t\geq 0 \textrm{ we have } t\boldv\in\mathcal{L}^M(\Sigma)\}.
\end{equation*}
\begin{proof}	
Let us consider $\boldv\in E^M(\Sigma)$ and a sequence $\{\boldv^k\}_{k=1}^\infty\subset L^\infty(\Sigma)$ be such that $\|\boldv^k-\boldv\|_{L^M}\rightarrow 0$ as $k\rightarrow\infty$. Then for an arbitrary $t\geq 0$ we obtain using the convexity of $M$ in the second variable that
\begin{equation*}
	\int_\Sigma M(x,t\boldv)\dx\leq\frac{1}{2}\left(\int_\Sigma M(x,2t\boldv^k)\dx+\int_\Sigma M(x,2t(\boldv-\boldv^k))\dx\right).
	\end{equation*}
	The first integral on the right hand side is finite by~\eqref{MLocBound} and the second one vanishes in the limit $k\rightarrow\infty$ by~\eqref{SSS}. Thus we showed that for all $t\geq 0\ t\boldv\in\mathcal{L}^M(\Sigma)$.\\
	Let $\boldv\in L^M(\Sigma)$ and assume that for all $t\geq 0$
	\begin{equation}\label{MultipleVInOC}
		 t\boldv\in\mathcal{L}^M(\Sigma).
	\end{equation}
	Defining $\boldv^k:=\boldv\chi_{\{|\boldv|\leq k\}}$, we have $\{\boldv^k\}_{k=1}^\infty\subset L^\infty(\Sigma)$ and due to~\eqref{MultipleVInOC} we get for all $t\geq 0$
	\begin{equation*}
		\int_\Sigma M(x,t(\boldv^k-\boldv))\dx= \int_{\{|\boldv|>k\}}M(x,t\boldv)\dx\overset{k\rightarrow\infty}{\rightarrow} 0.
	\end{equation*}
Hence for given $\delta>0$ we find $k_0(\delta)$ such that for any $k\geq k_0$ $\int_\Sigma M(x,\delta^{-1}(\boldv^k-\boldv))\dx<1$. Then we obtain $\|\boldv^k-\boldv\|_{L^M}\leq \delta$ by the definition of the Luxemburg norm and we conclude that $\boldv\in E^M(\Sigma)$.
\end{proof}
\end{Lemma}

The last statement of this subsection concerns possible relaxing of assumption \ref{M4}. Namely, we show that for an $\mathcal{N}$--function $M$ that is radially symmetric in the second variable the accomplishment of \ref{MTh} implies the validity of \ref{M4}.
\begin{Lemma}\label{Lem:M2CondRelax}
	Let $M:\eR^d\times[0,\infty)\rightarrow[0,\infty)$ be an $\mathcal{N}$--function satisfying conditions \ref{MTwo} and \ref{MTh}. Assume that $Q^\delta_j$ with $\delta<\delta_0:=\frac{1}{8\sqrt{d}}$ is an arbitrary cube defined in \ref{MTh} and that there are constants $A>0$ and $B\geq 1$ such that for all $y_1,y_2\in\Sigma$ (where either $\Sigma$ is a bounded Lipschitz domain or $\Sigma=Y$) with $|y_1-y_2|\leq \frac{1}{2}$ and all $\xi\in[0,\infty)$ the assumption \ref{MTh} holds. Then for $M^\delta_j$ given by \eqref{MDeltajDef} and its biconjugate $(M^\delta_j)^{**}$ it follows that \ref{M4} is satisfied.
	\begin{proof}
		First, we fix an arbitrary $y\in Q^\delta_j$ and note that
		\begin{equation}\label{Quotient}
			\frac{M(y,\xi)}{(M^\delta_j)^{**}(\xi)}=\frac{M(y,\xi)}{M^\delta_j(\xi)}\frac{M^\delta_j(\xi)}{(M^\delta_j)^{**}(\xi)}.
		\end{equation}
		We estimate separately both quotients on the right hand side of the latter equality. By continuity of $M$ we find $\bar{y}\in \tilde{Q}^\delta_j$ such that $M^\delta_j(\xi)=M(\bar{y},\xi)$. Then using condition \ref{MTh} and the fact that $|y-\bar{y}|\leq 3\delta\sqrt{d}<\frac{1}{2}$ we get
		\begin{equation}\label{FirQuoEst}
			\frac{M(y,\xi)}{M(\bar y,\xi)}\leq \max\{\xi^{-\frac{A}{\log|y-\bar y|}}, B^{-\frac{A}{\log|y-\bar y|}}\}\leq \max\{\xi^{-\frac{A}{\log(3\delta\sqrt{d})}}, B^{-\frac{A}{\log(3\delta\sqrt{d})}}\}.
		\end{equation}
		In order to estimate the second quotient in \eqref{Quotient} we observe first that if $\xi\in[0,\infty)$ is such that $M^\delta_j(\xi)=(M^\delta_j)^{**}(\xi)$ then the statement is obvious. Therefore we assume that $M^\delta_j(\xi_0)>(M^\delta_j)^{**}(\xi_0)$ at some $\xi_0$. Due to continuity of $M^\delta_j$ and $(M^\delta_j)^{**}$ there is a neighborhood $U$ of $\xi_0$ such that $M^\delta_j>(M^\delta_j)^{**}$ on $U$. Consequently, $(M^\delta_j)^{**}$ is affine on $U$. Moreover, \ref{MTwo} implies that $m_1\leq M^\delta_j\leq m_2$, where $m_1$ and $m_2$ are convex. Therefore there are $\xi_1,\xi_2$ such that $U\subset(\xi_1,\xi_2)$, $M^\delta_j>(M^\delta_j)^{**}$ on $(\xi_1,\xi_2)$, $(M^\delta_j)^{**}(\xi_i)=M^\delta_j(\xi_i)$, $i=1,2$ and $(M^\delta_j)^{**}$ is an affine function on $[\xi_1,\xi_2]$, i.e., for $t\in[0,1]$
\begin{equation}\label{ConvexificationIsAffine}
	(M^\delta_j)^{**}(t\xi_1+(1-t)\xi_2)=tM^\delta_j(\xi_1)+(1-t)M^\delta_j(\xi_2).
\end{equation}
We note that $\xi_1>0$ is always assumed because it follows that $0=M^\delta_j(0)=(M^\delta_j)^{**}(0)$. Now, thanks to the continuity of $M$ we find $y_i\in\tilde{Q}^\delta_j$ such that $M^\delta_j(\xi_i)=M(y_i,\xi_i)$, $i=1,2$. Consequently, it follows from \eqref{ConvexificationIsAffine} that
\begin{equation}\label{ConvexificationIsAffineII}
	(M^\delta_j)^{**}(t\xi_1+(1-t)\xi_2)=tM(y_1,\xi_1)+(1-t)M(y_2,\xi_2).
\end{equation}
Denoting $\tilde\xi=t\xi_1+(1-t)\xi_2$ we get
\begin{equation}\label{QuoBiConEst}
	\frac{M^\delta_j\left(\tilde\xi\right)}{(M^\delta_j)^{**}\left(\tilde\xi\right)}\leq \frac{M\left(y_2,\tilde\xi\right)}{tM(y_1,\xi_1)+(1-t)M(y_2,\xi_2)}\leq\frac{tM(y_2,\xi_1)+(1-t)M(y_2,\xi_2)}{tM(y_1,\xi_1)+(1-t)M(y_2,\xi_2)}.
\end{equation}
Next, we observe that the definition of $M^\delta_j$ implies $M(y_1,\xi_1)=M^\delta_j(\xi_1)\leq M(y_2,\xi_1)$. We can assume without loss of generality that
\begin{equation}\label{MXiOneIneq}
	M(y_1,\xi_1)< M(y_2,\xi_1)
\end{equation}
because for $M(y_1,\xi_1)= M(y_2,\xi_1)$ inequality \eqref{QuoBiConEst} implies $M^\delta_j\leq(M^\delta_j)^{**}$ on $[\xi_1,\xi_2]$. Since we have always $M^\delta_j\geq(M^\delta_j)^{**}$ we arrive at $M^\delta_j=(M^\delta_j)^{**}$ on $[\xi_1,\xi_2]$.

Let us consider a function $h:[0,1]\rightarrow\eR$ defined by
\begin{equation*}
	h(t)=\frac{tM(y_2,\xi_1)+(1-t)M(y_2,\xi_2)}{tM(y_1,\xi_1)+(1-t)M(y_2,\xi_2)}.
\end{equation*}
Then we compute
\begin{equation*}
	h'(t)=\frac{(M(y_2,\xi_1)-M(y_1,\xi_1))M(y_2,\xi_2)}{(t(M(y_1,\xi_1)-M(y_2,\xi_2))+M(y_2,\xi_2))^2}.
\end{equation*}
Obviously, we have $h'>0$ on $(0,1)$ due to \eqref{MXiOneIneq}. Therefore the maximum of $h$ is attained at $t=1$, which implies
\begin{equation}
	\frac{M^\delta_j\left(\tilde\xi\right)}{(M^\delta_j)^{**}\left(\tilde\xi\right)}\leq\frac{M(y_2,\xi_1)}{M(y_1,\xi_1)}.
\end{equation}
Next, we apply condition \ref{MTh} and $\xi_1\leq\tilde\xi$ to infer
\begin{equation}\label{SecQuoEst}
	\frac{M^\delta_j\left(\tilde\xi\right)}{(M^\delta_j)^{**}\left(\tilde\xi\right)}\leq \max\{\xi_1^{\frac{-A}{\log|y_2-y_1|}}, B^{\frac{-A}{\log|y_2-y_1|}}\}\leq \max\{\xi^{\frac{-A}{\log|y_2-y_1|}}, B^{\frac{-A}{\log|y_2-y_1|}}\}\leq\max\{\xi^{\frac{-A}{\log(4\delta\sqrt{d})}}, B^{\frac{-A}{\log(4\delta\sqrt{d})}}\}
\end{equation}
since $y_1,y_2\in\tilde{Q}^\delta_j$ implies $|y_1-y_2|\leq 4\delta\sqrt{d}<\frac{1}{2}$. Combining \eqref{Quotient} with \eqref{FirQuoEst} and \eqref{SecQuoEst} yields
\begin{equation*}
	\frac{M(y,\xi)}{(M^\delta_j)^{**}(\xi)}\leq \max\{\xi^{\frac{-A}{\log(3\delta\sqrt{d})}}, B^{\frac{-A}{\log(3\delta\sqrt{d})}}\}\cdot \max\{\xi^{\frac{-A}{\log(4\delta\sqrt{d})}}, B^{\frac{-A}{\log(4\delta\sqrt{d})}}\}\leq \max\{\xi^{\frac{-2A}{\log(4\delta\sqrt{d})}}, B^{\frac{-2A}{\log(4\delta\sqrt{d})}}\}
\end{equation*}
which is the desired conclusion.
\end{proof}
\end{Lemma}

\section{Auxiliary tools} \label{Ape2}
The first auxiliary tool is related to Young measures. The fundamental theorem on Young measures may be found in \cite{M96}. We only recall the lemma with properties of Young measures that will be used further. In the following $\mathcal{M}(\eR^d)$ stands for the space of bounded Radon measures on $\eR^d$.

\begin{Lemma}\cite[Corollary 3.2]{M96} \label{Lem:YMProp}
	Let a Young measure $\nu:\Omega\rightarrow \mathcal{M}(\eR^d)$ be generated by a sequence of measurable functions $\boldz^k:\Omega\rightarrow\eR^d$. Let $F:\Omega\times\eR^d\rightarrow\eR$ be a Carath\'eodory function. Let also assume that the negative part $F^-(\cdot,\boldz^k)$ is weakly relatively compact in $L^1(\Omega)$. Then
	\begin{equation*}
		\liminf_{k\rightarrow\infty}\int_\Omega F(x,\boldz^k(x))\dx\geq \int_\Omega\int_{\Rd} F(x,\boldzeta)\d\nu_x(\boldzeta)\dx.
	\end{equation*}
	If, in addition, the sequence of functions $x\mapsto |F|(x,\boldz^k(x))$ is weakly relatively compact in $L^1(\Omega)$ then
	\begin{equation*}
		F(\cdot,\boldz^k(\cdot))\rightharpoonup \int_{\eR^d}F(x,\boldzeta)\d\nu_x(\boldzeta)\dx.
	\end{equation*}
\end{Lemma}

Second result, we recall here is of functional analytic type.
\begin{Lemma}\label{Lem:Duality}
	Let $X$ be a Banach space, $V$ be a subspace of $X$, $g$ be a closed, convex functional on $X$ that is continuous at some $x\in V$.
	Then
	\begin{equation}\label{IdDual}
     \inf_{x\in V} \{g(x)-\langle \eta,x\rangle\}+\inf_{\xi\in V^\bot} g^*(\eta+\xi)=0
	\end{equation}
	for all  $\eta\in X^*$.
	\begin{proof}
	One deduces by definition of a convex conjugate that	
	\begin{equation}\label{GConj}
\forall\xi\in X^*: (g-\eta)^*(\xi)=\sup_{x\in X}\{\langle \eta+\xi,x\rangle-g(x)\}=g^*(\eta+\xi).
	\end{equation}
	According to \cite[Theorem 14.2]{ZKO94}
	\begin{equation*}
		\inf_{x\in V} A(x)+\inf_{x^*\in V^\bot}A^*(x^*)=0
	\end{equation*}
	for a closed, convex functional $A$ that is continuous at some $x\in V$. We set $A(x):=(g-\eta)(x)$ and the expression for $A^*$ determined by~\eqref{GConj} in the latter equality to conclude~\eqref{IdDual}.
	\end{proof}
\end{Lemma}

\section{Existence of solutions to elliptic problems}\label{Ape3}
To the best of authors' knowledge only the result from \cite{GMW12} concerns the existence of weak solutions of elliptic problems in which the growth condition is given by an anisotropic inhomogeneous ${\mathcal N}$--function. In \cite{GMW12},  only a scalar problem and an ${\mathcal N}$--function satisfying the condition~\ref{MTw2} are considered and for \ref{EMD} one could follow exactly the same procedure without the need of $L^{\infty}$ truncation. In this part we show that the result in \cite{GMW12} can be extended also to the vector valued problems provided we assume that the domain is star--shaped, i.e., the assumption~\ref{MTw} holds.

\begin{Theorem}\label{Thm:WSExistMHoeldC}
	Let $N\geq 1$, $\Omega\subset\Rd,d\geq 2$ be a star--shaped domain, an operator $\boldA$ satisfy~\ref{AO},\ref{ATh} and~\ref{AF}. Let an ${\mathcal N}$--function $M:\Omega\times\eR^{d\times N}\rightarrow[0,\infty)$  fulfill \ref{MTwo} and \ref{M4} with $\Omega$ replacing $Y$. Then the problem
	\begin{equation*}
		\begin{alignedat}{2}
			\div \boldA(x,\nabla \boldu(x))&=\div \boldF(x)&&\textrm{ in }\Omega,\\
			\boldu&=0&&\textrm{ on }\partial\Omega
		\end{alignedat}
	\end{equation*}
	possesses a unique weak solution, which is a function $\boldu\in V^M_0$ such that for all $\bphi\in V^M_0$
	\begin{equation}\label{MLogHWeakForm}
		\int_\Omega \boldA(x,\nabla \boldu) \cdot\nabla\bphi\dx=\int_\Omega \boldF\cdot\nabla\bphi\dx.
	\end{equation}
	\begin{proof}
		The construction of a weak solution $\boldu$ will be performed in several steps following the approach from \cite{GMW12}. First, we consider for $\delta\in(0,1)$ an auxiliary problem: to find $\boldu^\delta\in V^m_0$ such that
		\begin{equation}\label{ApprSys}
			\int_\Omega \boldB(x,\nabla \boldu^\delta(x))\cdot\nabla\bphi(x)\dx=\int_\Omega \boldF(x)\cdot\nabla\bphi(x)\dx\text{ for all }\bphi\in C^\infty_c(\Omega;\eR^N),
		\end{equation}
		where we denoted $\boldB(x,\boldzeta):=\boldA(x,\boldzeta)+\delta \NabOv m(\boldzeta)$ and $\NabOv m(\boldzeta):=\tilde{m}'(|\boldzeta|)\frac{\boldzeta}{|\boldzeta|}$. The ${\mathcal N}$--function $\tilde{m}:[0,\infty)\rightarrow[0,\infty)$ is such that $\tilde{m}^*$ satisfies $\Delta_2$--condition and $\tilde{m}(|\boldzeta|)\geq\sup_{x\in\Omega} M(x,\boldzeta)$. Moreover, the identity
		\begin{equation}\label{FYEqm}
			\NabOv m(\boldzeta)\cdot\boldzeta=m(\boldzeta)+m^*(\NabOv m(\boldzeta))
		\end{equation}
		holds.
		We show the existence of $\boldu^\delta$ and derive estimates of $\nabla \boldu^\delta$ and $\boldA(\cdot,\nabla \boldu^\delta)$ in $L^M(\Omega;\eR^{d\times N})$, $L^{M^*}(\Omega;\eR^{d\times N})$ respectively that are uniform with respect to $\delta\in(0,1)$. Having the uniform estimates we pass to the limit $\delta\rightarrow 0_+$ to obtain a weak solution of the initial problem. The reason for such a modification is that from now the leading $\mathcal{N}$--function is independent of the spatial variable and its conjugate satisfies $\Delta_2$--condition, which may not be the case in the original setting.\\
		\textbf{Step 1}:
			In order to obtain the existence of $\boldu^\delta$ for fixed $\delta\in(0,1)$ we employ the results on the so--called~$(S_m)$ class operators from \cite{MT99}. It is necessary to verify assumption of \cite[Theorem 4.3]{MT99}. We omit the verification since it is performed in the same manner as in the proof of \cite[Theorem 2.1]{GMW12}. The existence of a weak solution $\boldu^\delta$ of~\eqref{ApprSys} then follows by \cite[Theorem 5.1]{MT99}.\\
		\textbf{Step 2}: Now, we derive estimates uniform with respect to $\delta$. Since $\boldu^\delta\in V^m_0$, by Theorem~\ref{Thm:ModDensity} (claim 2) there is a sequence $\{\boldu^{\delta,k}\}_{k=1}^\infty\subset C^\infty_c(\Omega;\eR^N)$ such that $\nabla \boldu^{\delta,k}\ModConv{m} \nabla \boldu^\delta$ as $k\rightarrow\infty$. As $\boldu^{\delta,k}$ for each $k$ can be used as a test function in~\eqref{ApprSys}, Lemma~\ref{Lem:ProdConv} then implies
		\begin{equation}\label{st.point}
			\int_\Omega (\boldA(x,\nabla \boldu^\delta(x))+\delta\NabOv m(\nabla \boldu^\delta(x)))\cdot\nabla \boldu^\delta(x)\dx=\int_\Omega \boldF(x)\cdot\nabla \boldu^\delta(x)\dx.
		\end{equation}
		We get by~\ref{ATh}, \eqref{FYEqm}, the Young inequality using also the fact that $c\in (0,1]$ in~\ref{ATh} together with the convexity of $M$ with respect to the second variable that
		\begin{equation*}
			\int_{\Omega} \frac{c}{2}M(x,\nabla \boldu^\delta(x)) + c M^*(x,\boldA(x,\nabla \boldu^\delta(x)) + \delta m(\nabla \boldu^\delta(x))+ \delta m^*(\NabOv m(\boldu^\delta(x))\dx\leq \int_\Omega M^*\left(x,\frac{2}{c}\boldF(x)\right)\dx.
		\end{equation*}
		Hence we have
		\begin{equation}\label{AEDeltaAppr}
		\begin{split}
			\int_\Omega M(x,\nabla \boldu^\delta(x))\dx&\leq c,\\
			\int_\Omega M^*(x,\boldA(x,\nabla \boldu^\delta(x))\dx&\leq c,\\
			\int_\Omega \delta m^*(\NabOv m(\boldu^\delta(x))\dx&\leq c.
			\end{split}
		\end{equation}
		Consequently, we obtain the existence of a sequence $\{\delta_k\}_{k=1}^\infty$ such that $\delta_k\rightarrow 0$ as $k\rightarrow \infty$ and denoting $\boldA^k=\boldA(\cdot,\nabla \boldu^{\delta_k})$, $\uk=\boldu^{\delta_k}$ and $\boldB^k=\boldB(\cdot,\nabla \boldu^{\delta_k}(x))$ we have
		\begin{equation}\label{ConvApprDelta}
			\begin{alignedat}{2}
				\nabla \uk&\WSCon\nabla \boldu&&\text{ in }L^{M}(\Omega;\eR^{d\times N}),\\
				\boldA^k&\WSCon\bar\boldA&&\text{ in }L^{M^*}(\Omega; \eR^{d\times N}),\\
\boldB^k&\WCon \bar\boldA&&\text{ in }L^{1}(\Omega; \eR^{d\times N})
			\end{alignedat}
		\end{equation}
		as $k\rightarrow\infty$.\\
		\textbf{Step 3}:
We shall show that
		\begin{equation}\label{LimEq}
			\limsup_{k\rightarrow\infty}\int_\Omega \boldA^k\cdot\nabla \uk \dx\leq\int_\Omega \bar\boldA\cdot\nabla \boldu \dx.
		\end{equation}
Adding the limit $k\rightarrow\infty$ in~\eqref{st.point} with $\delta:=\delta_k$ and $\bphi:=\uk$ we get by using~\eqref{ConvApprDelta}$_1$ that
		\begin{equation}\label{LimitAppSTestSol}
			\limsup_{k\rightarrow\infty}\int_\Omega \boldA(x,\nabla \uk)\cdot\nabla \uk \dx \le \lim_{k\rightarrow\infty}\int_\Omega \boldB(x,\nabla \uk)\cdot\nabla \uk \dx=\lim_{k\rightarrow\infty} \int_\Omega \boldF\cdot\nabla\uk \dx=\int_\Omega \boldF\cdot\nabla\boldu \dx.
		\end{equation}
		Employing~\eqref{ConvApprDelta}$_3$ we can pass to the limit $k\rightarrow\infty$ in~\eqref{ApprSys} to obtain
\begin{equation}\label{ApprSys23}
			\int_\Omega \bar\boldA\cdot\nabla\bphi \dx=\int_\Omega \boldF\cdot\nabla\bphi\dx\qquad \text{ for all }\bphi\in C^\infty_c(\Omega;\eR^N).
		\end{equation}
Next, for any $\bphi\in V^M_0$ we can use the assumptions on the domain $\Omega$ and $M$, and by Theorem~\ref{Thm:ModDensity} (claim 2) find a sequence $\{\bphi^{k}\}_{k=1}^\infty\subset C^\infty_c(\Omega;\eR^N)$ such that $\nabla \bphi^{k}\ModConv{M} \nabla\bphi$ as $k\rightarrow\infty$. Thus, we can use $\bphi^k$ in~\eqref{ApprSys23} and by using Lemma~\ref{Lem:ProdConv}, we deduce the identity
\begin{equation}
			\int_\Omega \bar\boldA\cdot\nabla \bphi \dx=\int_\Omega \boldF\cdot\nabla \bphi \dx \text{ for all }\bphi\in V^M_0.\label{betteri}
\end{equation}
Inserting $\bphi:=\boldu$ into~\eqref{betteri} yields
		\begin{equation}\label{ResSysTestSol}
			\int_\Omega \bar\boldA\cdot\nabla \boldu \dx=\int_\Omega \boldF\cdot\nabla \boldu \dx.
		\end{equation}
		We conclude~\eqref{LimEq} by comparing~\eqref{LimitAppSTestSol} and~\eqref{ResSysTestSol}.\\
		\textbf{Step 4}: To finish the existence proof it remains to show that
		\begin{equation}\label{LimIden}
			\bar\boldA(x)=\boldA(x,\nabla \boldu(x))\text{ a.e. in }\Omega.
		\end{equation}
Indeed, once we have~\eqref{LimIden}, we can combine it with~\eqref{betteri} to obtain~\eqref{MLogHWeakForm}. Thus, we focus on~\eqref{LimIden}.
		Since $\boldA(\cdot,0)=0$ and $\boldA$ is strictly monotone, one sees immediately that the negative part of $\boldA(x,\nabla \uk)\cdot\nabla \uk$ vanishes. Thus it is relatively weakly compact in $L^1(\Omega)$. By the second part of Lemma~\ref{Lem:YMProp} we infer
		\begin{equation*}
		\liminf_{k\rightarrow\infty}\int_\Omega \boldA(x,\nabla \uk)\cdot\nabla \uk\dx\geq \int_\Omega\int_{\eR^{d\times N}}\boldA(x,\boldzeta)\cdot\boldzeta\d\nu_x(\boldzeta)\dx,
		\end{equation*}
		where $\nu_x$ is the Young measure generated by $\{\nabla \uk\}_{k=1}^\infty$. Comparing the latter inequality with~\eqref{LimEq} we have
		\begin{equation}\label{YMeasIneq}
			\int_\Omega\int_{\eR^{d\times N}}\boldA(x,\boldzeta)\cdot\boldzeta\d\nu_x(\boldzeta)\dx\leq \int_\Omega\bar\boldA\cdot\nabla \boldu\dx.
		\end{equation}
		Let us define $h(x,\boldzeta):=\left(\boldA(x,\boldzeta)-\boldA(x,\nabla \boldu)\right)\cdot\left(\boldzeta-\nabla \boldu\right)$. Then it follows from~\ref{AF} that
		\begin{equation}\label{IntHNonneg}
			\int_\Omega\int_{\eR^{d\times N}}h(x,\boldzeta)\d\nu_x(\boldzeta)\dx\geq 0.
		\end{equation}
		As $\boldA$ is a Carath\'eodory function and the sequences $\{\nabla \uk\}_{k=1}^\infty$ and $\{\boldA^k\}_{k=1}^\infty$ are weakly relatively compact in $L^1(\Omega)$ due to~\eqref{AEDeltaAppr}$_{1,2}$, the second part of Lemma~\ref{Lem:YMProp} implies
		\begin{equation}\label{LimitAsIntOfYMeas}
			\begin{alignedat}{2}
				\nabla \boldu&=\int_{\eR^{d\times N}}\boldzeta\d\nu_x(\boldzeta)&&\text{ a.e. in }\Omega,\\
				\bar\boldA&=\int_{\eR^{d\times N}}\boldA(x,\boldzeta)\d\nu_x(\boldzeta)&&\text{ a.e. in }\Omega.
			\end{alignedat}
		\end{equation}
		Using these identities we deduce that
		\begin{equation*}
		\int_\Omega\int_{\eR^{d\times N}} h(x,\boldzeta)\d\nu_x(\boldzeta)\dx=\int_\Omega \int_{\eR^{d\times N}}\boldA(x,\boldzeta)\cdot\boldzeta\d\nu_x(\boldzeta)\dx-\int_\Omega\bar\boldA\cdot\nabla \boldu\dx\leq 0
		\end{equation*}
		by~\eqref{YMeasIneq}. It follows from~\eqref{IntHNonneg} that $\int_{\eR^{d\times N}}h(x,\boldzeta)\d\nu_x(\boldzeta)=0$ a.e. in $\Omega$. Since $\nu_x$ is a probability measure and $\boldA(x,\cdot)$ is strictly monotone, we conclude for a.a. $x\in\Omega$ that  $\nu_x=\delta_{\nabla \boldu(x)}$ a.e. in $\Omega$ and inserting this into~\eqref{LimitAsIntOfYMeas}$_2$ we conclude~\eqref{LimIden}.\\
		\textbf{Step 5}: In order to show uniqueness of a weak solution, we suppose that functions $\boldu_1,\boldu_2\in V^M_0$ fulfill~\eqref{MLogHWeakForm}. Taking the difference of weak formulation with $\bphi:=\boldu_1-\boldu_2$ yields
		\begin{equation*}
			\int_\Omega \left(\boldA(x,\nabla \boldu_1)-\boldA(x,\nabla \boldu_2)\right)\cdot\nabla(\boldu_1-\boldu_2)\dx=0.
		\end{equation*}
Hence we obtain by~\ref{AF} that $\nabla (\boldu_1-\boldu_2)=0$ a.e. in $\Omega$ and since the trace of $\boldu_1-\boldu_2$ is zero on $\partial\Omega$ we conclude $\boldu_1=\boldu_2$ a.e. in $\Omega$.
	\end{proof}
\end{Theorem}

Similarly, as in the case of the monotone operator $\boldA$, we shall show certain properties of the minimizers to convex functional generated by the $\mathcal{N}$--function $M$. For simplicity, we state the following results only for spatially periodic setting, but they can be easily generalized also to the Dirichlet case. The main goal of the section is the following Lemma.
\begin{Lemma}\label{Lem:infjemin}
Let $M:Y\times \eR^{d\times N}\rightarrow[0,\infty)$ be an $\mathcal{N}$--function. Then for arbitrary $\boldxi \in \eR^{d\times N}$ there exists $\tilde\boldu\in V^M_{per}$ such that for all $\boldw\in V^M_{per}$ there holds
\begin{equation}
\int_Y M(y,\boldxi+\nabla \tilde\boldu(y))\dy \le \int_Y M(y,\boldxi+\nabla \boldw(y))\dy. \label{mindef}
\end{equation}
In addition, if $M$ is strictly convex then the minimizer is unique. Furthermore, if $M$ satisfies~\ref{MTh} then
\begin{equation}
\int_Y M(y,\boldxi+\nabla \tilde\boldu(y))\dy  = \inf_{\boldw\in W^{1}_{per}E^M(Y;\eR^N)}\int_Y M(y,\boldxi+\nabla \boldw(y))\dy. \label{mindef2}
\end{equation}
\end{Lemma}
\begin{proof}
The existence of a function $\tilde\boldu$ solving~\eqref{mindef} easily follows from the convexity of $M$ and the fact that $\int_{Y}M(y,\boldxi)\dy <\infty$. The uniqueness in case of the strict convexity is also a standard task. Thus, we focus only on~\eqref{mindef2}. We denote $\boldA(y,\boldxi):=\nabla_{\boldxi} M(y,\boldxi)$. Notice that due to the convexity $\boldA$ exists for almost all $\boldxi$ and we extend it to the whole $\eR^{d\times N}$ as a pseudodifferential. In addition, the operator $\boldA$ is a monotone mapping and there holds
$$
M(y,\boldxi) + M^*(y, \boldA(\boldxi)) = \boldA(\boldxi)\cdot \boldxi.
$$

Next, we use Theorem~\ref{Thm:WSExistMHoeldC} to get an existence of $\boldu\in V^M_{per}$, which solves for all $\boldw\in C^{\infty}_{per}(Y;\eR^N)$
\begin{equation}
\int_Y \boldA(y,\nabla \boldu+\boldxi) \cdot \nabla \boldw \dy=0. \label{kopl}
\end{equation}
Finally, due to the assumption on $M$ (namely the log-H\"{o}lder continuity~\ref{MTh}, we see from Theorem~\ref{Thm:ModDensity} that for any $\boldv\in V^M_{per}$ we can find a sequence $\{\boldv^n \}_{n=1}^\infty\in C^{\infty}_{per}(Y;\eR^N)$ that converges modularly to $\boldv$. Using the modular covergence we can set $\boldw:=\boldv^n$ in~\eqref{kopl}, which after letting $n\to \infty$ leads to
\begin{equation}
\int_Y \boldA(y,\nabla \boldu+\boldxi) \cdot \nabla \boldv \dy=0 \qquad \textrm{ for all } \boldv \in V^M_{per}\label{kopl2}
\end{equation}
and in particular to
$$
\int_Y \boldA(y,\nabla \boldu+\boldxi) \cdot (\nabla \boldu -\nabla \tilde\boldu) \dy=0.
$$
Hence, due to the convexity of $M$, we see that
$$
\int_{Y} M(y,\boldxi + \nabla \tilde\boldu) -M(y,\boldxi+ \nabla \boldu)\dy \ge \boldA(y,\nabla \boldu+\boldxi) \cdot (\nabla \tilde\boldu -\nabla\boldu) \dy=0.
$$
Therefore, $\boldu$ is also a minimizer to~\eqref{mindef}. In addition, following step by step the proof of Lemma~\ref{Lem:2.10}, we deduce that $\boldu$ can be constructed such that there is a sequence $\{\boldu^n\}_{n=1}^{\infty}\subset W^{1}_{per}E^M(Y;\eR^N)$ such that
$$
\int_{Y} M(y,\boldxi + \nabla \boldu)\dy = \lim_{n\to \infty} \int_{Y} M(y,\boldxi + \nabla \boldu^n)\dy.
$$
From this~\eqref{mindef2} directly follows.
\end{proof}
\end{appendix}

\section*{Acknowledgement}
M.~Bul\'{\i}\v{c}ek was partially supported
by the Czech Science Foundation (grant no. 16-03230S). The work of M.~Kalousek was supported by funds of the National Science Center awarded on the basis of decision No DEC-2013/09/D/ST1/03692. The research of A.~\'Swierczewska--Gwiazda and P.~Gwiazda have received funding from the National Science Centre, Poland, 2014/13/B/ST1/03094. This work was partially supported by the Simons - Foundation grant 346300 and the Polish Government MNiSW 2015-2019 matching fund.

\bibliographystyle{amsplain}
\bibliography{literature}

\end{document}